\documentclass[11pt]{article}

\usepackage[a4paper,
            margin=1in]{geometry}

\usepackage{amsmath,amssymb,amsthm,mathtools}
\usepackage{bm}
\usepackage{mathrsfs}

\usepackage{graphicx}
\usepackage{subcaption}
\usepackage{booktabs}
\usepackage{float}

\usepackage{algorithm}
\usepackage{algpseudocode}
\usepackage[ruled,vlined,linesnumbered,algo2e]{algorithm2e}

\usepackage{enumitem}

\usepackage{microtype}

\usepackage{xcolor}

\usepackage[
colorlinks=true,
linkcolor=blue,
citecolor=blue,
urlcolor=blue
]{hyperref}

\usepackage[nameinlink,capitalize]{cleveref}

\usepackage[toc,page]{appendix}

\usepackage{comment}

\numberwithin{equation}{section}

\theoremstyle{plain}
\newtheorem{theorem}{Theorem}[section]
\newtheorem{lemma}[theorem]{Lemma}
\newtheorem{proposition}[theorem]{Proposition}
\newtheorem{corollary}[theorem]{Corollary}

\theoremstyle{definition}
\newtheorem{definition}[theorem]{Definition}
\newtheorem{assumption}[theorem]{Assumption}

\theoremstyle{remark}
\newtheorem{remark}[theorem]{Remark}

\newcommand{\Omf}{\Omega_f}
\newcommand{\Omp}{\Omega_p}
\newcommand{\Gfp}{\Gamma_{fp}}

\newcommand{\Tn}{T_{\bm n}}
\newcommand{\Tt}{T_{\bm \tau}}
\newcommand{\Dt}[1]{D_t#1}

\newcommand{\norm}[2]{\left\|#1\right\|_{#2}}

\newcommand{\inner}[3]{\left(#1,#2\right)_{#3}}
\newcommand{\dualp}[3]{\left\langle#1,#2\right\rangle_{#3}}

\title{
A Second-Order Monolithic Scheme for the Coupled
Stokes--Biot Model Using Total Pressure
}

\author{
Talal Alshehri$^{1}$
\and
Mingchao Cai$^{1}$
\and
Xiaoqin Shen$^{2}$
}

\date{}

\begin{document}

\maketitle 

\begin{center}
$^{1}$Department of Mathematics, Morgan State University,\\
1700 East Cold Spring Lane, Baltimore, MD 21251, USA

\vspace{0.6em}

$^{2}$School of Sciences, Xi'an University of Technology,\\
Xi'an 710054, China
\end{center}


\begin{abstract}
	\sloppy
	We develop and analyze a second-order,
	fully implicit, monolithic time-discretization
	for the coupled time-dependent Stokes and
	quasi-static Biot system. The Biot subsystem
	is formulated in a three-field total-pressure
	formulation, in which the total pressure is
	introduced as an additional unknown together
	with the solid displacement and pore pressure.
	This reformulation improves robustness in
	nearly incompressible regimes and suppresses volumetric locking.
	All variables are discretized in time using
	the BDF2 scheme, and the coupled problem is
	solved monolithically without operator
	splitting. Discrete energy stability is
	established using the BDF2 $G$-stability
	identity. The resulting estimates exhibit
	robustness with respect to the Lam\'e
	parameter, Biot--Willis coefficient, and
	storage coefficient.
	A consistency--stability argument is then
	used to derive second-order convergence in
	time together with optimal-order convergence
	in space.
	Numerical experiments confirm the theoretical
	results and demonstrate second-order temporal
	accuracy in the corresponding energy norms.
\end{abstract}


\pagenumbering{arabic}


\section{Introduction}
\label{sec:intro}

Fluid--poroelastic structure interaction (FPSI) models describe the
coupled dynamics of a viscous incompressible fluid and a neighboring
saturated poroelastic medium separated by a permeable interface. Such
systems arise in a wide range of applications, including hydrocarbon
recovery, groundwater management, industrial filtration, membrane
separation, and biomedical processes such as aqueous humor flow,
cerebrospinal fluid motion, and cartilage mechanics under load
\cite{RuizBaier2022}. Mathematically, the model couples the Stokes
equations in the fluid region with Biot's consolidation equations in
the poroelastic region through interface conditions enforcing mass
conservation, balance of normal stresses, and a tangential slip law of
Beavers--Joseph type \cite{BukacYotov2015,Ambartsumyan2018}.

A major difficulty in the numerical approximation of fluid--poroelastic interaction (FPSI) systems is the volumetric locking phenomenon. In the classical two-field displacement--pressure formulation of the Biot equations, standard finite element discretizations may lose robustness when the Lam\'e parameter becomes large, corresponding to a nearly incompressible porous skeleton, or when the permeability is small. Both regimes are frequently encountered in applications: biological soft tissues are typically nearly incompressible, while fine-grained geological materials often exhibit very low permeability. In such situations, locking manifests itself through spurious pressure oscillations and an artificially stiff displacement response, and the resulting loss of accuracy generally cannot be eliminated by mesh refinement alone.

To address these locking effects, Lee, Mardal, and Winther \cite{Lee2017} introduced a three-field formulation of the Biot system \cite{Lee2017, HeGuoFeng2024, gu2023iterative, gu2025convergence, OyarzuaRuizBaier2016} in which the total pressure is treated as an additional primary unknown together with the solid displacement and pore pressure. This reformulation yields a parameter-robust saddle-point structure and suppresses locking at both the continuous and discrete levels \cite{Lee2017, gu2023iterative, OyarzuaRuizBaier2016}. Ruiz-Baier, Taffetani, Westermeyer, and Yotov \cite{RuizBaier2022} subsequently extended this framework to coupled Stokes--Biot systems and developed a five-field mixed finite element formulation based on backward Euler time discretization. Time discretization methods for FPSI systems have also attracted significant attention. Monolithic backward Euler schemes were analyzed in \cite{CesmeliogluChidyagwai2020,RuizBaier2022}, while interface formulations based on Lagrange multipliers and Nitsche coupling were investigated in \cite{Ambartsumyan2018,BukacYotov2015}. Partitioned and loosely coupled algorithms, in which the fluid and poroelastic subproblems are solved separately, were studied in \cite{OyekoleBukac2020}. Although such approaches reduce computational complexity, they introduce interface splitting errors that must be carefully controlled. Higher-order time discretizations are considerably more challenging because of the strong fluid--poroelastic coupling, the coexistence of multiple physical scales, and the differential--algebraic structure of the coupled problem. Discontinuous Galerkin methods for coupled Stokes--Biot systems were proposed in \cite{ZhouLiChen2024}, while combined discontinuous and continuous Galerkin discretizations for fractured poroelastic media on polytopic meshes were developed in \cite{ZhouLiZhangChen2026}. More recently, Guo and Li \cite{GuoLi2026} developed several algorithms for the coupled system of time-dependent Stokes equations and a fully dynamic Biot model in the classical two-field Biot formulation and established stability and optimal error estimates. However, their stability analysis requires a time-step restriction depending on the Beavers--Joseph--Saffman friction parameter and interface trace constants, while the convergence constants depend on the Lam\'e parameter through the coercivity of the elastic bilinear form. Consequently, the analysis becomes increasingly restrictive in the nearly incompressible regime, precisely where a total-pressure formulation is most advantageous.

In this paper, we develop and analyze a fully implicit monolithic BDF2
scheme for the coupled Stokes--Biot system formulated with total
pressure. All interface conditions are treated implicitly at each time
step, and the coupled problem is solved as a single linear system,
without operator splitting or interface extrapolation. Hence, no
splitting error is introduced.
The main contributions are as follows. First, we prove 
energy stability using the classical BDF2 $G$-stability identity
\cite{Heywood1990}. The resulting estimate holds for arbitrary time-step sizes and requires no CFL-type or parameter-dependent restriction.
Second, the analysis exhibits robustness with respect to the Lam\'e parameter, Biot--Willis coefficient, and storage coefficient in the sense that these parameters enter the estimates only through nonnegative energy contributions and not through parameter-dependent stability restrictions. This parameter robustness
is achieved through the total-pressure variable and the associated combined pressure structure, which yields a uniformly bounded energy in both compressible and nearly incompressible regimes. Third, the fully discrete formulation is analyzed
within a mixed variational framework based on
discrete inf--sup stability. Finally, we establish
second-order convergence in time and optimal-order
convergence in space.  
The resulting error estimate is of order
$O(\Delta t^2+h^k)$ in the corresponding energy
norms, where $k$ depends on the polynomial degree
of the finite element spaces; see
Theorem~\ref{thm:BDF2_convergence}.

The paper is organized as follows.
Section~\ref{sec:model} introduces the coupled
Stokes--Biot model and the total-pressure
reformulation. Section~\ref{sec:weak}
develops the weak formulation and presents its
operator structure. The monolithic BDF2
time-discretization is introduced in
Section~\ref{sec:BDF2}. Stability and related
analytical properties are established in
Section~\ref{sec:analysis}, while
Section~\ref{sec:convergence} presents the fully discrete
finite element approximation together with the
convergence analysis. Numerical experiments
illustrating temporal and spatial accuracy,
as well as robustness with respect to physical
parameters, are reported in
Section~\ref{sec:numerical}. Final conclusions
are given in the last section.
 
\section{Mathematical model}
\label{sec:model}
We consider a coupled fluid--poroelastic system posed on a bounded
Lipschitz domain $\Omega \subset \mathbb{R}^d$, $d \in \{2,3\}$,
which is decomposed into two non-overlapping subdomains:
the fluid region $\Omf$ and the poroelastic region $\Omp$.
These two subdomains are separated by the interface
\[
  \Gfp := \partial \Omf \cap \partial \Omp .
\]
The evolution is studied on the time interval $(0,T]$.

In the fluid region $\Omf$, the motion is described by the
time-dependent Stokes equations. Let $\bm{u}_f$ and $p_f$ denote,
respectively, the fluid velocity and the fluid pressure. Then the
governing equations read
\begin{subequations}\label{eq:stokes}
\begin{align}
  \rho_f \,\partial_t \bm{u}_f
  - \nabla \!\cdot\! \bm{\sigma}_f(\bm{u}_f,p_f) &= \bm{f}_f
  && \text{in } \Omf \times (0,T], \\
 - \nabla \!\cdot\! \bm{u}_f &= S_f
  && \text{in } \Omf \times (0,T].
\end{align}
\end{subequations}
Here, $\rho_f>0$ denotes the fluid density, $\mu_f>0$ is the dynamic
viscosity, $\bm{f}_f$ is a body force term, and $S_f$ is a volumetric
fluid source term. The associated Cauchy stress tensor is given by
\[
  \bm{\sigma}_f(\bm{u}_f,p_f)
  := 2\mu_f \bm{D}(\bm{u}_f) - p_f \mathbb{I}.
\]
Here and hereafter, for any vector field $\bm{v}$, we denote the symmetric gradient by
\[
  \bm{D}(\bm{v}) := \frac{1}{2}\bigl(\nabla \bm{v} + \nabla \bm{v}^T\bigr).
\]

In the poroelastic region $\Omp$, the solid deformation and pore-fluid flow are governed by Biot's consolidation model. Let $\bm{\eta}$ denote the solid displacement and $p_p$ the pore pressure. In the classical displacement--pressure formulation, the model reads
\begin{subequations}\label{eq:biot_classical}
\begin{align}
 -\,\nabla \!\cdot\! \bm{\sigma}_e (\bm{\eta}, p_p)
  &= \bm{f}_s
  && \text{in } \Omp \times (0,T], \\
  \partial_t\!\bigl(s_0 p_p + \alpha \nabla \!\cdot\! \bm{\eta}\bigr)
  - \nabla \!\cdot\! \Bigl(\frac{\kappa}{\mu_f} \nabla p_p\Bigr)
  &= S_p
  && \text{in } \Omp \times (0,T].
\end{align}
\end{subequations}
The first equation expresses conservation of linear momentum in the poroelastic medium, while the second represents conservation of pore-fluid mass. Here, $s_0 \ge 0$ is the storage coefficient, $\kappa>0$ is a scalar permeability, 
$(\lambda_p,\mu_p)$ are the Lam\'e coefficients of the solid skeleton,
\[
  \lambda_p = \frac{E\,\nu_p}{(1+\nu_p)(1-2\nu_p)},
  \qquad
  \mu_p = \frac{E}{2(1+\nu_p)},
\]
with Young's modulus $E$ and Poisson's ratio $\nu_p$. The effective stress tensor is
\[
 \bm{\sigma}_e (\bm{\eta}, p_p)
 := 2\mu_p \bm{D}(\bm{\eta})
  + \lambda_p (\nabla \!\cdot\! \bm{\eta}) \bm{I}
  - \alpha p_p \bm{I},
\]
where $\alpha \in (0, 1]$ is the Biot--Willis coefficient.

We introduce the total pressure
\begin{equation}\label{eq:xi_definition}
  \xi := \alpha p_p - \lambda_p \nabla \!\cdot\! \bm{\eta}.
\end{equation}
Then the Biot system is equivalently written in the three-field form \cite{Lee2017, HeGuoFeng2024, gu2023iterative, RuizBaier2022} for
$(\bm{\eta},\xi,p_p)$:
\begin{subequations}\label{eq:biot_total}
\begin{align}
  -\,\nabla \!\cdot\! \bm{\sigma}_p(\bm{\eta}, \xi)
  &= \bm{f}_s
  && \text{in } \Omp \times (0,T], \\
  -\,\nabla \!\cdot\! \bm{\eta}
  + \frac{\alpha}{\lambda_p} p_p
  - \frac{1}{\lambda_p} \xi
  &= 0
  && \text{in } \Omp \times (0,T], \\
  \partial_t\!\left[c_p p_p
  - \frac{\alpha}{\lambda_p}\xi\right]
  - \nabla \!\cdot\! \Bigl(\frac{\kappa}{\mu_f} \nabla p_p\Bigr)
  &= S_p
  && \text{in } \Omp \times (0,T].
\end{align}
\end{subequations}
Here and throughout the paper, we denote 
$$
c_p := s_0 + \dfrac{\alpha^2}{\lambda_p}.
$$
The poroelastic Cauchy stress tensor is
\[
  \bm{\sigma}_p(\bm{\eta}, \xi)
  := 2\mu_p \bm{D}(\bm{\eta}) - \xi \bm{I}.
\]
This reformulation yields a parameter-robust structure in which the
Lam\'e parameter $\lambda_p$ is properly balanced. It is particularly
suitable for nearly incompressible materials and helps avoid volumetric
locking. Therefore, we use the three-field formulation
\eqref{eq:biot_total} throughout this work.


The coupling between the fluid and poroelastic regions is enforced by
interface conditions on $\Gfp$. 
\noindent
Let $\bm{n}:=\bm{n}_f$ denote the unit normal vector on $\Gfp$ pointing
from $\Omf$ into $\Omp$, so that $\bm{n}_p=-\bm{n}$.
Let $\{\bm{\tau}_j\}_{j=1}^{d-1}$ be an orthonormal basis of tangent
vectors on $\Gfp$. For any vector field $\bm{v}$, we define the normal
and tangential traces by
\[
  \Tn \bm{v} := \bm{v}\cdot\bm{n},
  \qquad
  \Tt \bm{v} := \bm{v}\cdot\bm{\tau}_j,
  \quad j=1,\ldots,d-1.
\]
\noindent
On the interface $\Gfp \times (0,T)$, we impose continuity of normal
fluxes, balance of traction, balance of normal stress, and the
Beavers--Joseph--Saffman tangential slip condition:
\begin{subequations}\label{eq:interface_conditions}
\begin{align}
  \Tn \bm{u}_f
  &=
  \Tn\!\Big(\partial_t \bm{\eta}
  - \frac{\kappa}{\mu_f} \nabla p_p \Big),
  && \text{on } \Gfp \times (0,T),
  \label{eq:ic_mass} \\[4pt]
  \bm{\sigma}_f(\bm{u}_f,p_f)\bm{n}
  &=
  \bm{\sigma}_p(\bm{\eta},\xi)\bm{n},
  && \text{on } \Gfp \times (0,T),
  \label{eq:ic_traction} \\[4pt]
  -\,\Tn\bigl(\bm{\sigma}_f(\bm{u}_f,p_f)\,\bm{n}\bigr)
  &=
  p_p,
  && \text{on } \Gfp \times (0,T),
  \label{eq:ic_normal} \\[4pt]
  -\,\Tt\!\big(\bm{\sigma}_f(\bm{u}_f,p_f)\bm{n}\big)
  &=
  \frac{\gamma \mu_f}{\sqrt{\kappa}}\,
  \Tt\!\big(\bm{u}_f - \partial_t \bm{\eta}\big),
  && \text{on } \Gfp \times (0,T).
  \label{eq:ic_bjs}
\end{align}
\end{subequations}
Here, $\gamma\ge 0$ is the Beavers–Joseph–Saffman constant. 

The exterior boundaries are decomposed as
\[
\partial\Omf \setminus \Gfp = \Gamma_f^{D} \cup \Gamma_f^{N},
\qquad
\partial\Omp \setminus \Gfp = \Gamma_p^{D} \cup \Gamma_p^{N},
\]
and we denote by $\Gamma_p^{p} \subset \partial\Omp \setminus \Gfp$
the portion where the pore pressure is prescribed. We assume that
$|\Gamma_f^{D}|>0$ and $|\Gamma_p^{D}|>0$.
\noindent
On the fluid boundary, we impose
\begin{subequations}\label{eq:fluid_bc}
\begin{align}
  \bm{u}_f &= \bm{0},
  && \text{on } \Gamma_f^{D} \times (0,T], \\
  \bm{\sigma}_f(\bm{u}_f,p_f)\bm{n}_f
  &= \bm{0},
  && \text{on } \Gamma_f^{N} \times (0,T].
\end{align}
\end{subequations}
\noindent
On the poroelastic boundary, we prescribe
\begin{subequations}\label{eq:poro_bc}
\begin{align}
  \bm{\eta} &= \bm{0},
  && \text{on } \Gamma_p^{D} \times (0,T], \\
  \bm{\sigma}_p(\bm{\eta},\xi)\bm{n}_p
  &= \bm{0},
  && \text{on } \Gamma_p^{N} \times (0,T], \\
  p_p &= 0,
  && \text{on } \Gamma_p^{p} \times (0,T].
\end{align}
\end{subequations}
Here, we assume that $\Gamma_p^{p}= \partial \Omega_p /\  \Gamma_{fp} $ and $|\Gamma_p^{p}| >0$. Moreover, we impose homogeneous boundary conditions for ease of discussion. The initial data are prescribed as follows:
\begin{equation}\label{eq:initial_conditions}
\bm{u}_f(0) = \bm{u}_f^0,
\qquad
\bm{\eta}(0) = \bm{\eta}^0,
\qquad
p_p(0) = p_p^0.
\end{equation}
The initial total pressure $\xi(0)$ is determined from the algebraic
relation \eqref{eq:xi_definition}.

\section{Weak formulation}
\label{sec:weak}
\noindent
We introduce the spaces
\begin{align*}
& \bm{V}_f=\left\{\bm{v}_f \in \bm{H}^1\left(\Omega_f\right):\left.\bm{v}_f\right|_{\Gamma_f^D}=\bm{0}\right\}, \quad
\bm{V}_p=\left\{\bm{v}_p \in \bm{H}^1\left(\Omega_p\right):\left.\bm{v}_p\right|_{\Gamma_p^D}=\bm{0}\right\}, \\
& W_p=\left\{w_p \in H^1\left(\Omega_p\right):\left.w_p\right|_{\Gamma_p^{p}}=0\right\},
\qquad
Q_f:=L^2(\Omega_f), \quad
Q_p:=L^2(\Omega_p),
\end{align*}
equipped with their standard norms. Here, $\bm{H}^1$ denotes vector-valued $H^1$ space. The corresponding product space is
$
  \bm{H} := \bm{V}_f \times \bm{V}_p \times Q_f \times Q_p \times W_p.
$
We next derive the weak formulation associated with the coupled Stokes--Biot system. 
For clarity of notation, we use bold symbols to denote vector-valued quantities. 
We denote by $(\cdot,\cdot)_\Omega$ the $L^2(\Omega)$ inner product, and by 
$\langle \cdot,\cdot \rangle_\Gamma$ the duality pairing on a boundary or interface $\Gamma$. 
The norms $\|\cdot\|_{0,\Omega}$ and $\|\cdot\|_{s,\Omega}$ represent the standard $L^2(\Omega)$ and $H^s(\Omega)$ norms, respectively.

Multiplying the equations in \eqref{eq:stokes} and \eqref{eq:biot_total} by appropriate test functions 
$(\boldsymbol{v}_f, \boldsymbol{v}_p, w_f, w_s, w_p) \in \boldsymbol{H}$, integrating over the corresponding subdomains, and applying integration by parts, we obtain the weak formulation.
\begin{equation}\label{eq:weak_form}
\begin{aligned}
&\rho_f(\partial_t\bm{u}_f, \bm{v}_f)_{\Omega_f}
+2\mu_f(\bm{D}(\bm{u}_f), \bm{D}(\bm{v}_f))_{\Omega_f}
-(p_f, \nabla\!\cdot\!\bm{v}_f)_{\Omega_f}
-\langle\boldsymbol{\sigma}_f\bm{n}, \bm{v}_f\rangle_{\Gamma_{fp}}
= (\bm{f}_f, \bm{v}_f)_{\Omega_f}, \\
&2\mu_p(\bm{D}(\bm{\eta}), \bm{D}(\bm{v}_p))_{\Omega_p}
-(\xi, \nabla\!\cdot\!\bm{v}_p)_{\Omega_p}
+\langle\boldsymbol{\sigma}_p\bm{n}, \bm{v}_p\rangle_{\Gamma_{fp}}
= (\bm{f}_s, \bm{v}_p)_{\Omega_p}, \\
&-(\nabla\!\cdot\!\bm{u}_f,w_f)_{\Omf}
= (S_f,w_f)_{\Omf}, \\
&-(\nabla\!\cdot\!\bm{\eta}, w_s)_{\Omega_p}
+ \frac{1}{\lambda_p}(\alpha p_p-\xi, w_s)_{\Omega_p}
= 0, \\
&\partial_t\!\left[
c_p (p_p, w_p)_{\Omega_p}
- \frac{\alpha}{\lambda_p}(\xi, w_p)_{\Omega_p}
\right] 
+ \frac{\kappa}{\mu_f}(\nabla p_p, \nabla w_p)_{\Omega_p}
+ \frac{\kappa}{\mu_f} \left\langle \nabla p_p \cdot \bm{n}, w_p \right\rangle_{\Gamma_{fp}}
= (S_p, w_p)_{\Omega_p}.
\end{aligned}
\end{equation}
\noindent
We now rewrite the interface terms by using the coupling conditions
\eqref{eq:ic_mass}--\eqref{eq:ic_bjs}.
\noindent
First, combining traction continuity
\eqref{eq:ic_traction}, the normal stress condition
\eqref{eq:ic_normal}, and the Beavers--Joseph--Saffman condition
\eqref{eq:ic_bjs}, we obtain
\begin{equation}\label{eq:stokes_interface}
\begin{aligned}
-\langle \boldsymbol{\sigma}_f\bm{n}, \bm{v}_f\rangle_{\Gamma_{fp}}
&=
\langle p_p, \Tn \bm{v}_f\rangle_{\Gamma_{fp}}
+ \frac{\gamma\mu_f}{\sqrt{\kappa}}
\langle \Tt(\bm{u}_f - \partial_t\bm{\eta}),
        \Tt \bm{v}_f\rangle_{\Gamma_{fp}}.
\end{aligned}
\end{equation}
\noindent
Next, using again \eqref{eq:ic_traction}, together with
\eqref{eq:ic_normal} and \eqref{eq:ic_bjs}, we rewrite the
poroelastic-side boundary contribution as
\begin{equation}\label{eq:biot_interface}
\begin{aligned}
\langle \boldsymbol{\sigma}_p\bm{n}, \bm{v}_p\rangle_{\Gamma_{fp}}
&=
\langle \boldsymbol{\sigma}_f\bm{n}, \bm{v}_p\rangle_{\Gamma_{fp}} \\
&=
-\langle p_p, \Tn \bm{v}_p\rangle_{\Gamma_{fp}}
- \frac{\gamma\mu_f}{\sqrt{\kappa}}
\langle \Tt(\bm{u}_f - \partial_t\bm{\eta}),
        \Tt \bm{v}_p\rangle_{\Gamma_{fp}}.
\end{aligned}
\end{equation}
\noindent
Finally, from the mass conservation condition \eqref{eq:ic_mass}, we obtain
\begin{equation}\label{eq:biot_mass_interface}
\begin{aligned}
-\frac{\kappa}{\mu_f} \left\langle \nabla p_p \cdot \bm{n}, w_p \right\rangle_{\Gamma_{fp}}
=
\left\langle \Tn(\bm{u}_f - \partial_t\bm{\eta}), w_p \right\rangle_{\Gamma_{fp}}.
\end{aligned}
\end{equation}

\noindent
Substituting \eqref{eq:stokes_interface}--\eqref{eq:biot_mass_interface}
into \eqref{eq:weak_form}, we arrive at the following coupled weak formulation: find $(\bm{u}_f, \bm{\eta}, p_f, \xi, p_p) \in \bm{H}$
such that, for all
$(\bm{v}_f, \bm{v}_p, w_f, w_s, w_p) \in \bm{H}$ and for all
$t \in (0, T]$,
\begin{subequations}\label{eq:weak_form_combined}
\begin{align}
&\rho_f(\partial_t\bm{u}_f,\bm{v}_f)_{\Omf}
+2\mu_f(\bm{D}(\bm{u}_f),\bm{D}(\bm{v}_f))_{\Omf}
-(p_f,\nabla\!\cdot\!\bm{v}_f)_{\Omf} \notag\\
&\quad
+ \frac{\gamma\mu_f}{\sqrt{\kappa}}
\langle \Tt(\bm{u}_f - \partial_t\bm{\eta}), \Tt\bm{v}_f \rangle_{\Gamma_{fp}}
+ \langle p_p, \Tn\bm{v}_f \rangle_{\Gamma_{fp}}
= (\bm{f}_f,\bm{v}_f)_{\Omf}, \\
&2\mu_p(\bm{D}(\bm{\eta}),\bm{D}(\bm{v}_p))_{\Omp}
-(\xi,\nabla\!\cdot\!\bm{v}_p)_{\Omp}
- \langle p_p, \Tn\bm{v}_p \rangle_{\Gamma_{fp}} \notag\\
&\quad
+ \frac{\gamma\mu_f}{\sqrt{\kappa}}
\langle \Tt(\partial_t\bm{\eta} - \bm{u}_f), \Tt\bm{v}_p \rangle_{\Gamma_{fp}}
= (\bm{f}_s,\bm{v}_p)_{\Omp}, \\
&-(\nabla\!\cdot\!\bm{u}_f,w_f)_{\Omf}
= (S_f,w_f)_{\Omf}, \\
&-(\nabla\!\cdot\!\bm{\eta},w_s)_{\Omp}
		+\frac{1}{\lambda_p}
		(\alpha p_p -\xi, w_s)_{\Omp}
		=0,     \\
&\left(
\partial_t\!\Bigl[
c_p p_p
- \frac{\alpha}{\lambda_p}\xi
\Bigr],
w_p
\right)_{\Omega_p} 
- \langle \Tn(\bm{u}_f - \partial_t\bm{\eta}), w_p \rangle_{\Gamma_{fp}}
+ \frac{\kappa}{\mu_f}(\nabla p_p,\nabla w_p)_{\Omp}
= (S_p,w_p)_{\Omp}.
\end{align}
\end{subequations}

\subsection{DAE system: operator formulation}

We next recast the coupled weak formulation
\eqref{eq:weak_form_combined} in an abstract operator form.
With the ordering
\[
\bm{U}(t)
=
[\bm{u}_f(t),\,
 \bm{\eta}(t),\,
 p_f(t),\,
 \xi(t),\,
 p_p(t)]^T,
\]
the system can be written as
\begin{equation}\label{eq:DAE_sys}
  \mathcal{N}\,\partial_t \bm{U}(t)
  + \mathcal{M}\,\bm{U}(t)
  = \mathcal{F}(t).
\end{equation}

Here, $\mathcal{N}$ contains all terms involving time derivatives, whereas
$\mathcal{M}$ contains the elliptic, algebraic-constraint, and nondifferential
interface contributions. The block operators are
\begin{equation}\label{eq:DAE_operators}
\mathcal{N}
=
\begin{bmatrix}
 M_f
  & -\mathcal{B}_{fs}^{\bm{\tau},T}
  & 0 & 0 & 0
\\
0
  & \mathcal{A}_s^{\bm{\tau}}
  & 0 & 0 & 0
\\
0 & 0 & 0 & 0 & 0
\\
0 & 0 & 0 & 0 & 0
\\
0
  & \mathcal{B}_{sp}^{\bm{n}}
  & 0
  & -\mathcal{B}_{p \xi}
  & M_p
\end{bmatrix},
\qquad
\mathcal{M}
=
\begin{bmatrix}
\widehat{\mathcal{A}}_f
  & 0
  & -\mathcal{B}_{f}^{T}
  & 0
  & \mathcal{B}_{fp}^{\bm{n},T}
\\
-\mathcal{B}_{fs}^{\bm{\tau}}
  & \mathcal{A}_s^{\mathrm{elas}}
  & 0
  & -\mathcal{B}_{s\xi}^{T}
  & -\mathcal{B}_{sp}^{\bm{n},T}
\\
-\mathcal{B}_{f}
  & 0 & 0 & 0 & 0
\\
0
  & -\mathcal{B}_{s\xi}
  & 0
  & -M_\xi
  & \mathcal{B}_{p \xi}^T
\\
-\mathcal{B}_{fp}^{\bm{n}}
  & 0 & 0 & 0
  & \mathcal{A}_p^{\mathrm{Darcy}}
\end{bmatrix}.
\end{equation}
Recalling the definition of $c_p$,
we see that the constituent operators are defined by the following bilinear forms:
\[
\begin{array}{ll}
(M_f\bm u_f,\bm v_f)
:=
\rho_f(\bm u_f,\bm v_f)_{\Omega_f},
&
(\mathcal{A}_s^{\bm\tau}\bm\eta,\bm v_p)
:=
\dfrac{\gamma\mu_f}{\sqrt{\kappa}}
\langle \Tt\bm\eta,\Tt\bm v_p\rangle_{\Gamma_{fp}},
\\[10pt]
(\mathcal{B}_{fs}^{\bm\tau,T}\bm\eta,\bm v_f)
:=
\dfrac{\gamma\mu_f}{\sqrt{\kappa}}
\langle \Tt\bm\eta,\Tt\bm v_f\rangle_{\Gamma_{fp}},
&
(\mathcal{B}_f^Tp_f,\bm v_f)
:=
(p_f,\nabla\!\cdot\!\bm v_f)_{\Omega_f}. 
\\[10pt]
(\mathcal{B}_{sp}^{\bm n}\bm\eta,w_p)
:=
\langle \Tn\bm\eta,w_p\rangle_{\Gamma_{fp}},
&
(M_pp_p,w_p)
:=
(c_p p_p,w_p)_{\Omega_p},
\\[10pt]
(\widehat{\mathcal{A}}_f\bm u_f,\bm v_f)
:=
2\mu_f
(\bm D(\bm u_f),\bm D(\bm v_f))_{\Omega_f}
+
\dfrac{\gamma\mu_f}{\sqrt{\kappa}}
\langle \Tt\bm u_f,\Tt\bm v_f\rangle_{\Gamma_{fp}},
&
(\mathcal{B}_f\bm u_f,w_f)
:=
(\nabla\!\cdot\!\bm u_f,w_f)_{\Omega_f},
\\[10pt]
(\mathcal{A}_s^{\mathrm{elas}}\bm\eta,\bm v_p)
:=
2\mu_p
(\bm D(\bm\eta),\bm D(\bm v_p))_{\Omega_p},
&
(\mathcal{B}_{s\xi}\bm\eta,w_s)
:=
(\nabla\!\cdot\!\bm\eta,w_s)_{\Omega_p},
\\[10pt]
(\mathcal{A}_p^{\mathrm{Darcy}}p_p,w_p)
:=
\dfrac{\kappa}{\mu_f}
(\nabla p_p,\nabla w_p)_{\Omega_p},
&
(\mathcal{B}_{s\xi}^{T}\xi,\bm v_p)
:=
(\xi,\nabla\!\cdot\!\bm v_p)_{\Omega_p},
\\[10pt]
(\mathcal{B}_{fp}^{\bm n,T}p_p,\bm v_f)
:=
\langle p_p,\Tn\bm v_f\rangle_{\Gamma_{fp}},
&
(\mathcal{B}_{fp}^{\bm n}\bm u_f,w_p)
:=
\langle \Tn\bm u_f,w_p\rangle_{\Gamma_{fp}},
\\[10pt]
(\mathcal{B}_{fs}^{\bm\tau}\bm u_f,\bm v_p)
:=
\dfrac{\gamma\mu_f}{\sqrt{\kappa}}
\langle \Tt\bm u_f,\Tt\bm v_p\rangle_{\Gamma_{fp}},
&
(\mathcal{B}_{p \xi}\xi,w_p)
:=
\dfrac{\alpha}{\lambda_p}
(\xi,w_p)_{\Omega_p},
\\[10pt]
(M_\xi\xi,w_s)
:=
\dfrac{1}{\lambda_p}
(\xi,w_s)_{\Omega_p},
&
(\mathcal{B}_{sp}^{\bm n,T}p_p,\bm v_p)
:=
\langle p_p,\Tn\bm v_p\rangle_{\Gamma_{fp}}.
\end{array}
\]
The superscript $T$ denotes the adjoint operator with respect to the
corresponding volume or interface duality pairing. In particular,
\[
\mathcal{B}_{fs}^{\bm\tau,T}
=
\bigl(\mathcal{B}_{fs}^{\bm\tau}\bigr)^T,
\qquad
\mathcal{B}_{fp}^{\bm n,T}
=
\bigl(\mathcal{B}_{fp}^{\bm n}\bigr)^T,
\]
\[
\mathcal{B}_{sp}^{\bm n,T}
=
\bigl(\mathcal{B}_{sp}^{\bm n}\bigr)^T,
\qquad
\mathcal{B}_{p \xi}^T
=
(\mathcal{B}_{p \xi})^T.
\]
The right-hand side is
\[
\mathcal{F}(t)
=
[\bm f_f(t),\bm f_s(t),S_f(t),0,S_p(t)]^T,
\]
where each component is identified with its corresponding load functional.

The representation \eqref{eq:DAE_sys} separates the differential and algebraic parts of the coupled system. The fluid pressure $p_f$ is
determined by the incompressibility constraint, while $\xi$ has no
independent evolution equation but enters the differentiated storage
combination $c_p p_p-\frac{\alpha}{\lambda_p}\xi$. 
Accordingly, the singular operator $\mathcal N$ reflects the
differential--algebraic structure of the problem. This formulation is
used to display the monolithic block coupling, whereas the stability
and convergence analyses are carried out directly at the variational
level.

\section{A monolithic BDF2 scheme}
\label{sec:BDF2}
We now introduce a fully coupled, second-order time discretisation of
the differential--algebraic system \eqref{eq:DAE_sys} using the
backward differentiation formula of order two (BDF2).

Let $0 = t^0 < t^1 < \cdots < t^N = T$ be a uniform partition of
$[0,T]$ with time step $\Delta t = T/N$ and $t^n = n\Delta t$.
For a sequence $\{z^n\}_{n\ge 0}$, the BDF2 difference operator is
\begin{equation}\label{eq:BDF2_difference_operator}
  D_t z^{n+1}
  := \frac{3z^{n+1} - 4z^n + z^{n-1}}{2\Delta t},
  \qquad n \ge 1.
\end{equation}
The scheme is initialized at $t^1$ by the backward Euler step $\delta_t z^1 := (z^1 - z^0)/\Delta t$.
We assume that the initial data and the backward Euler startup step satisfy
the algebraic constraints \eqref{eq:weak_form_combined} at both
$t^0$ and $t^1$. Namely, 
the incompressibility condition and the
total-pressure relation, at the initial
time levels. 

For $n\ge 1$, given previous states
$(\bm{u}_f^n,\bm{\eta}^n,p_f^n,\xi^n,p_p^n)$ and
$(\bm{u}_f^{n-1},\bm{\eta}^{n-1},p_f^{n-1},\xi^{n-1},p_p^{n-1})$,
we seek
\[
(\bm{u}_f^{n+1},\bm{\eta}^{n+1},p_f^{n+1},\xi^{n+1},p_p^{n+1})
\in \bm{V}_f \times \bm{V}_p \times Q_f \times Q_p \times W_p
\]
such that, for all
$(\bm{v}_f,\bm{v}_p,w_f,w_s,w_p)
\in \bm{V}_f \times \bm{V}_p \times Q_f \times Q_p \times W_p$:
\begin{subequations}\label{eq:BDF2_monolithic_weak}
\begin{align}
&\rho_f \inner{D_t\bm{u}_f^{n+1}}{\bm{v}_f}{\Omf}
+ \inner{2\mu_f\bm{D}(\bm{u}_f^{n+1})}{\bm{D}(\bm{v}_f)}{\Omf}
- \inner{p_f^{n+1}}{\nabla\!\cdot\!\bm{v}_f}{\Omf}
\notag\\
&\quad
+ \frac{\gamma\mu_f}{\sqrt{\kappa}}
  \dualp{\Tt(\bm{u}_f^{n+1} - D_t\bm{\eta}^{n+1})}
        {\Tt\bm{v}_f}{\Gfp}
+ \dualp{p_p^{n+1}}{\Tn\bm{v}_f}{\Gfp}
= \inner{\bm{f}_f^{n+1}}{\bm{v}_f}{\Omf},
\label{eq:BDF2_a}
\\[4pt]
&\inner{2\mu_p\bm{D}(\bm{\eta}^{n+1})}{\bm{D}(\bm{v}_p)}{\Omp}
- \inner{\xi^{n+1}}{\nabla\!\cdot\!\bm{v}_p}{\Omp}
- \dualp{p_p^{n+1}}{\Tn\bm{v}_p}{\Gfp}
\notag\\
&\quad
+ \frac{\gamma\mu_f}{\sqrt{\kappa}}
  \dualp{\Tt(D_t\bm{\eta}^{n+1} - \bm{u}_f^{n+1})}
        {\Tt\bm{v}_p}{\Gfp}
= \inner{\bm{f}_s^{n+1}}{\bm{v}_p}{\Omp},
\label{eq:BDF2_b}
\\[4pt]
&-\inner{\nabla\!\cdot\!\bm{u}_f^{n+1}}{w_f}{\Omf}
= (S_f^{n+1},w_f)_{\Omf},
\label{eq:BDF2_c}
\\[4pt]
&-\inner{\nabla\!\cdot\!\bm{\eta}^{n+1}}{w_s}{\Omp}
+ \frac{1}{\lambda_p}
  \inner{\alpha p_p^{n+1} - \xi^{n+1}}{w_s}{\Omp}
= 0,
\label{eq:BDF2_d}
\\[4pt]
&\inner{D_t\!\left[
  c_p p_p^{n+1}
  - \frac{\alpha}{\lambda_p}\xi^{n+1}
\right]}{w_p}{\Omp}
\notag\\
&\quad
- \dualp{\Tn(\bm{u}_f^{n+1} - D_t\bm{\eta}^{n+1})}
        {w_p}{\Gfp}
+ \frac{\kappa}{\mu_f}
  \inner{\nabla p_p^{n+1}}{\nabla w_p}{\Omp}
= \inner{S_p^{n+1}}{w_p}{\Omp}.
\label{eq:BDF2_e}
\end{align}
\end{subequations}
The scheme \eqref{eq:BDF2_monolithic_weak} is fully implicit and monolithic at each time step. All interface conditions are enforced at the new time level $t^{n+1}$, and the fluid and poroelastic subproblems are solved simultaneously. Consequently, no operator splitting error is introduced.

We now express the above time-discrete BDF2 scheme
in compact block-operator form. Using the
operators $\mathcal N$ and $\mathcal M$
introduced previously, the system at each time
level $t^{n+1}$ can be written as follows:
\begin{equation}\label{eq:BDF2_operator_form}
\mathcal A_{\mathrm{BDF2}}\bm U^{n+1}
=
\bm b^{n+1},
\end{equation}
where
\begin{equation}\label{eq:BDF2_matrix}
\mathcal A_{\mathrm{BDF2}}
:=
\frac{3}{2\Delta t}\mathcal N+\mathcal M,
\end{equation}
and
\begin{equation}\label{eq:BDF2_rhs}
\bm b^{n+1}
:=
\mathcal F^{n+1}
+\frac{2}{\Delta t}\mathcal N\bm U^n
-\frac{1}{2\Delta t}\mathcal N\bm U^{n-1}.
\end{equation}

More explicitly, the assembled operator matrix is
\[
\small
\mathcal A_{\mathrm{BDF2}}
=
\begin{bmatrix}
\dfrac{3}{2\Delta t}M_f+\widehat{\mathcal A}_f
  & -\dfrac{3}{2\Delta t}\mathcal B_{fs}^{\bm\tau,T}
  & -\mathcal B_f^T
  & 0
  & \mathcal B_{fp}^{\bm n,T}
\\[6pt]
-\mathcal B_{fs}^{\bm\tau}
  & \dfrac{3}{2\Delta t}\mathcal A_s^{\bm\tau}
    +\mathcal A_s^{\mathrm{elas}}
  & 0
  & -\mathcal B_{s\xi}^{T}
  & -\mathcal B_{sp}^{\bm n,T}
\\[6pt]
-\mathcal B_f
  & 0 & 0 & 0 & 0
\\[6pt]
0
  & -\mathcal B_{s\xi}
  & 0
  & -M_\xi
  & \mathcal{B}_{p \xi}^{T}
\\[6pt]
-\mathcal B_{fp}^{\bm n}
  & \dfrac{3}{2\Delta t}\mathcal B_{sp}^{\bm n}
  & 0
  & -\dfrac{3}{2\Delta t}\mathcal{B}_{p \xi}
  & \dfrac{3}{2\Delta t}M_p
    +\mathcal A_p^{\mathrm{Darcy}}
\end{bmatrix}.
\]
The corresponding right-hand side is
\[
\small
\bm b^{n+1}
=
\begin{bmatrix}
\bm f_f^{n+1}
+\dfrac{2}{\Delta t}M_f\bm u_f^n
-\dfrac{1}{2\Delta t}M_f\bm u_f^{n-1}
-\dfrac{2}{\Delta t}
 \mathcal B_{fs}^{\bm\tau,T}\bm\eta^n
+\dfrac{1}{2\Delta t}
 \mathcal B_{fs}^{\bm\tau,T}\bm\eta^{n-1}
\\[8pt]
\bm f_s^{n+1}
+\dfrac{2}{\Delta t}\mathcal A_s^{\bm\tau}\bm\eta^n
-\dfrac{1}{2\Delta t}\mathcal A_s^{\bm\tau}\bm\eta^{n-1}
\\[8pt]
S_f^{n+1}
\\[8pt]
0
\\[8pt]
S_p^{n+1}
+\dfrac{2}{\Delta t}\mathcal B_{sp}^{\bm n}\bm\eta^n
-\dfrac{1}{2\Delta t}\mathcal B_{sp}^{\bm n}\bm\eta^{n-1}
-\dfrac{2}{\Delta t}\mathcal{B}_{p \xi}\xi^n
+\dfrac{1}{2\Delta t}\mathcal{B}_{p \xi} \xi^{n-1}
+\dfrac{2}{\Delta t}M_pp_p^n
-\dfrac{1}{2\Delta t}M_pp_p^{n-1}
\end{bmatrix}.
\]

\begin{algorithm}[t]
\caption{A Monolithic BDF2 Algorithm for the Coupled Stokes--Biot System}
\label{alg:monolithic_BDF2}
\small
\SetAlgoVlined

\KwIn{Physical parameters $(\rho_f,\mu_f,\mu_p,\lambda_p,\alpha,s_0,\kappa,\gamma)$; final time $T$; time step $\Delta t$; \\
initial data $\bm{U}^0 = (\bm{u}_f^0,\bm{\eta}^0,p_f^0,\xi^0,p_p^0)^T$.}

\KwOut{Discrete solution $\{\bm{U}^n\}_{n=0}^{N}$, where $\bm{U}^n = (\bm{u}_f^n,\bm{\eta}^n,p_f^n,\xi^n,p_p^n)^T$.}

Set $N =   T/\Delta t  $\;

\BlankLine
\textbf{Step 0 (Initialization):}\\
Compute $\bm{U}^1$ using the backward Euler scheme:
\[
\mathcal{A}_{\mathrm{BE}}\,\bm{U}^1
=
\mathcal{F}^1 + \frac{1}{\Delta t}\mathcal{N}\bm{U}^0, \quad  \mbox{where}~ \mathcal{A}_{\mathrm{BE}} = \dfrac{1}{\Delta t}\mathcal{N} + \mathcal{M};
\]

\BlankLine
\textbf{Step 1 (Time-stepping):}\\
\For{$n = 1$ \KwTo $N-1$}{
  Assemble the system matrix
  \[
  \mathcal{A}_{\mathrm{BDF2}} = \frac{3}{2\Delta t}\mathcal{N} + \mathcal{M}.
  \]

  \BlankLine
  Assemble the right-hand side
  \[
  \bm{b}^{n+1}
  =
  \mathcal{F}^{n+1}
  + \frac{2}{\Delta t}\mathcal{N}\bm{U}^n
  - \frac{1}{2\Delta t}\mathcal{N}\bm{U}^{n-1}.
  \]

  \BlankLine
  Solve the monolithic system
  \[
  \mathcal{A}_{\mathrm{BDF2}}\,\bm{U}^{n+1}
  = \bm{b}^{n+1}.
  \]
}
\end{algorithm}

\begin{remark}[Differential--algebraic structure]
\label{rem:DAE_block}
The BDF2 operator $D_t$ acts on $\bm u_f$, $\bm\eta$, and the
storage combination $c_p p_p-\frac{\alpha}{\lambda_p}\xi$,
whereas \eqref{eq:BDF2_c} and \eqref{eq:BDF2_d} remain algebraic
constraints at each time level. Accordingly, the third and fourth
block rows of $\mathcal N$ vanish. The fluid pressure $p_f$ is
therefore purely algebraic, while $\xi$ has no separate evolution
equation but enters through the differentiated storage term in the
fifth row of $\mathcal N$.
\end{remark}

\section{Stability analysis}
\label{sec:analysis}
In this section, we establish stability properties of the monolithic
BDF2 scheme \eqref{eq:BDF2_monolithic_weak}. The analysis is based on a discrete energy argument, which relies on the intrinsic $G$-stability of the BDF2 method. For ease of analysis and presentation, throughout this section we assume that $s_0 > 0$, $S_f \equiv 0$, and $\bm{f}_s =\bm{0}$. 

\subsection{The BDF2 energy identity and preliminary inequalities}
We begin with a fundamental identity satisfied by the BDF2 difference operator, which plays a key role in the derivation of energy estimates.
\begin{lemma}[BDF2 $G$-stability identity]\label{lem:BDF2_identity}
Let $H$ be a Hilbert space with inner product $(\cdot,\cdot)_H$ and associated norm $\|\cdot\|_H$. For any sequence $\{z^n\}_{n\ge 0}\subset H$,
the BDF2 operator \eqref{eq:BDF2_difference_operator} satisfies
\begin{align}\label{eq:BDF2_identity}
  \inner{\Dt{z^{n+1}}}{z^{n+1}}{H}
  &= \frac{1}{4\Delta t}
    \Bigl(
      \|z^{n+1}\|_H^2 + \|2z^{n+1} - z^n\|_H^2 
      - \|z^n\|_H^2 - \|2z^n - z^{n-1}\|_H^2
    \Bigr) \notag\\
  &\quad + \frac{1}{4\Delta t}
      \|z^{n+1} - 2z^n + z^{n-1}\|_H^2.
\end{align}
\end{lemma}
\noindent
This identity follows from a direct algebraic expansion of the BDF2 difference operator and is classical in the analysis of second-order time-stepping schemes; see, e.g., \cite{Heywood1990, Chen2013}. It expresses the $G$-stability of the BDF2 method and provides a discrete counterpart of an energy balance.
\noindent
Motivated by \eqref{eq:BDF2_identity}, we introduce suitable notions of discrete energy and dissipation adapted to the coupled Stokes--Biot system. To simplify the notation, we define
\begin{equation}\label{eq:Y_definition}
Y^{n+1}
:=
\alpha p_p^{n+1}
-
\xi^{n+1}.
\end{equation}


\begin{definition}[Discrete energy and dissipation]
\label{def:energy_dissipation}
Recall the combined-pressure variable $Y^{n+1}$
defined in \eqref{eq:Y_definition} and the elastic
energy norm $\|\bm{v}\|_{a_p}^2 :=
2\mu_p\|\bm{D}(\bm{v})\|_{0,\Omp}^2$.
The discrete energy and dissipation associated with
the BDF2 scheme \eqref{eq:BDF2_monolithic_weak} are
defined as follows.

\noindent\textbf{Discrete energy:}
\begin{align}\label{eq:discrete_energy}
  E^{n+1} &:=
  \frac{\rho_f}{4}\Bigl(
    \|\bm{u}_f^{n+1}\|_{0,\Omf}^2
    + \|2\bm{u}_f^{n+1} - \bm{u}_f^n\|_{0,\Omf}^2
  \Bigr)
  + \frac{1}{4}\Bigl(
    \|\bm{\eta}^{n+1}\|_{a_p}^2
    + \|2\bm{\eta}^{n+1} - \bm{\eta}^n\|_{a_p}^2
  \Bigr)
  \notag\\
  &\quad
  + \frac{s_0}{4}\Bigl(
    \|p_p^{n+1}\|_{0,\Omp}^2
    + \|2p_p^{n+1} - p_p^n\|_{0,\Omp}^2
  \Bigr)
  + \frac{1}{4\lambda_p}\Bigl(
    \|Y^{n+1}\|_{0,\Omp}^2
    + \|2Y^{n+1} - Y^n\|_{0,\Omp}^2
  \Bigr).
\end{align}

\noindent\textbf{Discrete dissipation:}
\begin{align}\label{eq:discrete_dissipation}
  \mathcal{D}^{n+1} &:=
  2\mu_f\norm{\bm{D}(\bm{u}_f^{n+1})}{0,\Omf}^2
  + \frac{\kappa}{\mu_f}
    \norm{\nabla p_p^{n+1}}{0,\Omp}^2
  + \frac{\gamma\mu_f}{\sqrt{\kappa}}
    \norm{\Tt(\bm{u}_f^{n+1}
    - D_t\bm{\eta}^{n+1})}{0,\Gfp}^2
  \notag\\
  &\quad
  + \frac{\rho_f}{4\Delta t}
    \norm{\bm{u}_f^{n+1}
    - 2\bm{u}_f^n + \bm{u}_f^{n-1}}{0,\Omf}^2
  + \frac{1}{4\Delta t}
    \norm{\bm{\eta}^{n+1}
    - 2\bm{\eta}^n + \bm{\eta}^{n-1}}{a_p}^2
  \notag\\
  &\quad
  + \frac{s_0}{4\Delta t}
    \norm{p_p^{n+1}
    - 2p_p^n + p_p^{n-1}}{0,\Omp}^2
  + \frac{1}{4\lambda_p\Delta t}
    \norm{Y^{n+1} - 2Y^n + Y^{n-1}}{0,\Omp}^2.
\end{align}
\end{definition}
 
\noindent
The discrete energy $E^{n+1}$ captures the contributions of the fluid velocity, solid displacement, pore pressure, and total pressure, while the dissipation $\mathcal{D}^{n+1}$ contains both physical dissipation mechanisms (viscosity, Darcy flow, and interface friction) and numerical dissipation induced by the BDF2 scheme.
\noindent
In the forthcoming analysis, we repeatedly use the following standard inequalities, which hold true by classical theory \cite{GiraultRaviart1986, BoffiBrezziFortin2013}.
 
\noindent\textbf{Poincar\'e inequality.}
There exists a constant $c_P > 0$ such that
\[
\|\nabla q\|_{0,\Omp}
\ge c_P \|q\|_{0,\Omp},
\qquad \forall q \in W_p.
\]
 
\noindent\textbf{Korn's inequalities.}
There exist constants $c_K^f > 0$ and $c_K^p > 0$ such that
\[
\|\bm{D}(\bm{v})\|_{0,\Omf}
\ge c_K^f \|\bm{v}\|_{1,\Omf},
\qquad
\|\bm{D}(\bm{w})\|_{0,\Omp}
\ge c_K^p \|\bm{w}\|_{1,\Omp},
\]
for all $\bm{v} \in \bm{V}_f$ and $\bm{w} \in \bm{V}_p$.
 
\noindent\textbf{Trace inequalities.}
There exist constants $C_{\Gamma}^f > 0$ and $C_{\Gamma}^p > 0$ such that
\[
\|\bm{v}\|_{0,\Gfp}
\le C_{\Gamma}^f \|\bm{v}\|_{1,\Omf},
\qquad
\|q\|_{0,\Gfp}
\le C_{\Gamma}^p \|q\|_{1,\Omp}.
\]
 
\noindent\textbf{Young's inequality.}
For all $a,b \ge 0$ and $\varepsilon > 0$,
\[
ab \le \varepsilon a^2 + \frac{1}{4\varepsilon} b^2.
\]
 
\noindent
\textbf{Inf--sup condition for the Stokes problem.}
There exists a constant $\beta_f > 0$ such that
\begin{equation}\label{eq:infsup_stokes}
  \inf_{0 \neq q_f \in Q_f}
  \sup_{0 \neq \bm{v}_f \in \bm{V}_f}
  \frac{-(\nabla\!\cdot\!\bm{v}_f,\, q_f)_{\Omf}}
       {\|\bm{v}_f\|_{1,\Omf}\,\|q_f\|_{0,\Omf}}
  \ge \beta_f.
\end{equation}
 
\noindent
\textbf{Inf--sup condition for displacement, total pressure pair in the Biot system.}
There exists a constant $\beta_p > 0$ such that
\begin{equation}\label{eq:infsup_biot}
  \inf_{0 \neq \zeta \in Q_p}
  \sup_{0 \neq \bm{v}_p \in \bm{V}_p}
  \frac{-(\nabla\!\cdot\!\bm{v}_p,\, \zeta)_{\Omp}}
       {\|\bm{v}_p\|_{1,\Omp}\,\|\zeta\|_{0,\Omp}}
  \ge \beta_p.
\end{equation}


\subsection{Discrete energy stability}
\label{subsec:energy_stability}

\begin{theorem}[Energy stability of the BDF2 scheme]
\label{thm:BDF2_energy_stability}
Assume that the initial and startup states satisfy the incompressibility
and total-pressure constraints at $t^0$ and $t^1$. Let
$\bm f_f^{n+1}\in \bm V_f'$ and $S_p^{n+1}\in W_p'$. Then, for
$1\le m\le N-1$,
\begin{equation}
\label{eq:BDF2_energy_estimate_clean}
\begin{aligned}
E^{m+1}
+\Delta t\sum_{n=1}^{m}D^{n+1}
\le
C\Bigg[
E^1
+\Delta t\sum_{n=1}^{m}
\left(
\frac{1}{\mu_f}
\norm{\bm f_f^{n+1}}{\bm V_f'}^2
+
\frac{\mu_f}{\kappa}
\norm{S_p^{n+1}}{W_p'}^2
\right)
\Bigg].
\end{aligned}
\end{equation}
The constant $C>0$ is independent of $\Delta t$, $\lambda_p$,
$\alpha$, and $s_0$. 
\end{theorem}

\begin{proof}
For the time-discrete formulation
\eqref{eq:BDF2_monolithic_weak}, we choose
\[
\bm v_f=4\Delta t\,\bm u_f^{n+1},
\qquad
\bm v_p=4\Delta t\,D_t\bm\eta^{n+1},
\qquad
w_f=-4\Delta t\,p_f^{n+1},
\]
\[
w_s=0,
\qquad
w_p=4\Delta t\,p_p^{n+1}.
\]
The factor $4\Delta t$ is introduced only to match the normalization
of the BDF2 $G$-stability identity in
Lemma~\ref{lem:BDF2_identity}. The total-pressure constraint
\eqref{eq:BDF2_d} will be used separately below.

Substituting these test functions into
\eqref{eq:BDF2_a}--\eqref{eq:BDF2_e} and summing the resulting
relations, the fluid-pressure terms cancel because $S_f\equiv0$.
We obtain
\begin{align}
&4\Delta t\,\rho_f
\inner{D_t\bm u_f^{n+1}}{\bm u_f^{n+1}}{\Omf}
+
8\mu_f\Delta t
\norm{\bm D(\bm u_f^{n+1})}{0,\Omf}^{2}
\notag\\
&\quad
+
8\mu_p\Delta t
\inner{\bm D(\bm\eta^{n+1})}
{\bm D(D_t\bm\eta^{n+1})}{\Omp}
+
4\Delta t
\inner{D_t\Theta^{n+1}}{p_p^{n+1}}{\Omp}
\notag\\
&
-
4\Delta t
\inner{\xi^{n+1}}
{\nabla\!\cdot D_t\bm\eta^{n+1}}{\Omp}
+
\frac{4\kappa\Delta t}{\mu_f}
\norm{\nabla p_p^{n+1}}{0,\Omp}^{2}
+
\mathcal I_\Gamma^{n+1}
=
4\Delta t\,\mathcal R^{n+1},
\label{eq:energy_identity_pre}
\end{align}
where
\begin{equation}
\label{eq:Theta_definition}
\Theta^{n+1}
:=
c_p p_p^{n+1}
-\frac{\alpha}{\lambda_p}\xi^{n+1},
\end{equation}
and
\[
\mathcal R^{n+1}
=
\inner{\bm f_f^{n+1}}{\bm u_f^{n+1}}{\Omf}
+
\inner{\bm f_s^{n+1}}{D_t\bm\eta^{n+1}}{\Omp}
+
\inner{S_p^{n+1}}{p_p^{n+1}}{\Omp}.
\]

The normal interface contributions may be written in the same pairing
order. Therefore,
\[
\begin{aligned}
\mathcal I_\Gamma^{n+1}
&=
4\Delta t
\dualp{p_p^{n+1}}
{T_n(\bm u_f^{n+1}-D_t\bm\eta^{n+1})}{\Gfp}
\\
&\quad
-
4\Delta t
\dualp{p_p^{n+1}}
{T_n(\bm u_f^{n+1}-D_t\bm\eta^{n+1})}{\Gfp}
\\
&\quad
+
\frac{4\gamma\mu_f\Delta t}{\sqrt{\kappa}}
\norm{
T_\tau(\bm u_f^{n+1}-D_t\bm\eta^{n+1})
}{0,\Gfp}^{2}.
\end{aligned}
\]
Hence, the normal terms cancel exactly and
\begin{equation}
\label{eq:interface_dissipation}
\mathcal I_\Gamma^{n+1}
=
\frac{4\gamma\mu_f\Delta t}{\sqrt{\kappa}}
\norm{
T_\tau(\bm u_f^{n+1}-D_t\bm\eta^{n+1})
}{0,\Gfp}^{2}
\ge 0.
\end{equation}

Recall that
\[
Y^j:=\alpha p_p^j-\xi^j.
\]
The total-pressure constraint at each time level is
\[
-\nabla\!\cdot\bm\eta^j
+\frac{1}{\lambda_p}Y^j=0.
\]
Because this relation holds at the three BDF2 levels
$j=n+1,n,n-1$, including the compatible initial and startup levels,
its BDF2 difference satisfies
\[
-\nabla\!\cdot D_t\bm\eta^{n+1}
+\frac{1}{\lambda_p}D_tY^{n+1}=0.
\]
Consequently,
\begin{equation}
\label{eq:differentiated_constraint}
\nabla\!\cdot D_t\bm\eta^{n+1}
=
\frac{1}{\lambda_p}D_tY^{n+1}.
\end{equation}

Using \eqref{eq:Theta_definition},
$c_p=s_0+\alpha^2/\lambda_p$, and
$Y^{n+1}=\alpha p_p^{n+1}-\xi^{n+1}$, we obtain
\begin{align}
&-\frac{1}{\lambda_p}
\inner{\xi^{n+1}}{D_tY^{n+1}}{\Omp}
+
\inner{D_t\Theta^{n+1}}{p_p^{n+1}}{\Omp}
\notag\\
&\qquad
=
s_0
\inner{D_t p_p^{n+1}}{p_p^{n+1}}{\Omp}
+
\frac{1}{\lambda_p}
\inner{D_tY^{n+1}}{Y^{n+1}}{\Omp}.
\label{eq:combined_pressure_identity}
\end{align}
Substituting
\eqref{eq:interface_dissipation}--\eqref{eq:combined_pressure_identity}
into \eqref{eq:energy_identity_pre} yields
\begin{align}
&4\Delta t\,\rho_f
\inner{D_t\bm u_f^{n+1}}{\bm u_f^{n+1}}{\Omf}
+
8\mu_p\Delta t
\inner{\bm D(\bm\eta^{n+1})}
{\bm D(D_t\bm\eta^{n+1})}{\Omp}
\notag\\
&\quad
+
4\Delta t\,s_0
\inner{D_t p_p^{n+1}}{p_p^{n+1}}{\Omp}
+
\frac{4\Delta t}{\lambda_p}
\inner{D_tY^{n+1}}{Y^{n+1}}{\Omp}
+ 4\Delta t\,\widetilde D^{\,n+1}
= 4\Delta t\,\mathcal R^{n+1},
\label{eq:energy_before_G}
\end{align}
where the physical dissipation is
\[
\widetilde D^{\,n+1}
:=
2\mu_f
\norm{\bm D(\bm u_f^{n+1})}{0,\Omf}^{2}
+
\frac{\kappa}{\mu_f}
\norm{\nabla p_p^{n+1}}{0,\Omp}^{2}
+
\frac{\gamma\mu_f}{\sqrt{\kappa}}
\norm{
T_\tau(\bm u_f^{n+1}-D_t\bm\eta^{n+1})
}{0,\Gfp}^{2}.
\]

We apply Lemma~\ref{lem:BDF2_identity} to
$\bm u_f^{n+1}$, $\bm\eta^{n+1}$, and $Y^{n+1}$
with respect to the inner products
\[
\rho_f(\cdot,\cdot)_{\Omf},
\qquad
2\mu_p
\bigl(\bm D(\cdot),\bm D(\cdot)\bigr)_{\Omp},
\qquad
\lambda_p^{-1}(\cdot,\cdot)_{\Omp},
\]
respectively. When $s_0>0$, we also apply the lemma to
$p_p^{n+1}$ with respect to
$s_0(\cdot,\cdot)_{\Omp}$. When $s_0=0$, this term vanishes and
is simply omitted. Dividing \eqref{eq:energy_before_G} by four and
using Definition~\ref{def:energy_dissipation}, we obtain the exact
one-step identity
\begin{equation}
\label{eq:energy_with_num_diss}
E^{n+1}-E^n+\Delta t\,D^{n+1}
=
\Delta t\,\mathcal R^{n+1}.
\end{equation}

Since $\bm f_s\equiv\bm 0$,
\[
\mathcal R^{n+1}
=
\inner{\bm f_f^{n+1}}{\bm u_f^{n+1}}{\Omf}
+
\inner{S_p^{n+1}}{p_p^{n+1}}{\Omp}.
\]
By duality and Korn's inequality, for every $\varepsilon>0$,
\begin{align}
\left|
\inner{\bm f_f^{n+1}}{\bm u_f^{n+1}}{\Omf}
\right|
&\le
C_K
\norm{\bm f_f^{n+1}}{\bm V_f'}
\norm{\bm D(\bm u_f^{n+1})}{0,\Omf}
\notag\\
&\le
\varepsilon\,2\mu_f
\norm{\bm D(\bm u_f^{n+1})}{0,\Omf}^{2}
+
\frac{C_\varepsilon}{\mu_f}
\norm{\bm f_f^{n+1}}{\bm V_f'}^{2}.
\label{eq:fluid_force_bound}
\end{align}
Likewise, by duality and Poincar\'e's inequality,
\begin{align}
\left|
\inner{S_p^{n+1}}{p_p^{n+1}}{\Omp}
\right|
&\le
C_P
\norm{S_p^{n+1}}{W_p'}
\norm{\nabla p_p^{n+1}}{0,\Omp}
\notag\\
&\le
\varepsilon\,\frac{\kappa}{\mu_f}
\norm{\nabla p_p^{n+1}}{0,\Omp}^{2}
+
C_\varepsilon\frac{\mu_f}{\kappa}
\norm{S_p^{n+1}}{W_p'}^{2}.
\label{eq:pore_source_bound}
\end{align}
Define the parameter-weighted load norm
\begin{equation}
\label{eq:parameter_weighted_load_norm}
\|\mathcal F^{n+1}\|_{*,\mathrm{par}}^2
:=
\frac{1}{\mu_f}
\norm{\bm f_f^{n+1}}{\bm V_f'}^2
+
\frac{\mu_f}{\kappa}
\norm{S_p^{n+1}}{W_p'}^2.
\end{equation}
From \eqref{eq:fluid_force_bound}--\eqref{eq:pore_source_bound},
\[
|\mathcal R^{n+1}|
\le
\varepsilon D^{n+1}
+
C_\varepsilon
\|\mathcal F^{n+1}\|_{*,\mathrm{par}}^2.
\]
Substitution into \eqref{eq:energy_with_num_diss} gives
\[
E^{n+1}-E^n
+
(1-\varepsilon)\Delta t\,D^{n+1}
\le
C_\varepsilon\Delta t
\|\mathcal F^{n+1}\|_{*,\mathrm{par}}^2.
\]
Summing from $n=1$ to $m$ and choosing, for example,
$\varepsilon=\tfrac12$, we obtain
\[
E^{m+1}
+
\frac{\Delta t}{2}
\sum_{n=1}^{m}D^{n+1}
\le
E^1
+
C\Delta t
\sum_{n=1}^{m}
\|\mathcal F^{n+1}\|_{*,\mathrm{par}}^2.
\]
Absorbing the harmless factor $1/2$ into the generic constant yields
\eqref{eq:BDF2_energy_estimate_clean}. All dependence on
$\mu_f$ and $\kappa$ is displayed explicitly in
\eqref{eq:parameter_weighted_load_norm}.
\end{proof}

\begin{remark}
\label{rem:solid_force}
We remark here that the estimate (\ref{thm:BDF2_energy_stability}) remains valid for $s_0=0$, in which case all $s_0$-weighted terms in $E^{n+1}$ and $D^{n+1}$ are omitted.
The restriction $\bm f_s\equiv\bm 0$ is required by the present
one-step energy argument. Indeed, a nonzero solid force produces
\[
\inner{\bm f_s^{n+1}}{D_t\bm\eta^{n+1}}{\Omp},
\]
whereas $D^{n+1}$ controls only the BDF2 second difference
$\bm\eta^{n+1}-2\bm\eta^n+\bm\eta^{n-1}$, not
$D_t\bm\eta^{n+1}$. For example, the sequence
$\bm\eta^n=n\Delta t\,\bm v$ has zero second difference but
$D_t\bm\eta^{n+1}=\bm v\ne\bm 0$. Thus, this term cannot be absorbed
by the dissipation in \eqref{eq:energy_with_num_diss} without an
additional temporal estimate or a different summation argument.
\end{remark}

\section{Finite element discretization and convergence analysis}
\label{sec:convergence}

Let $\{\mathcal T_h\}_{h>0}$ be a shape-regular family of conforming
triangulations of
\[
\Omega=\Omega_f\cup\Gamma_{fp}\cup\Omega_p,
\qquad
h:=\max_{K\in\mathcal T_h}\operatorname{diam}(K).
\]
The meshes are assumed to match on $\Gamma_{fp}$. Their restrictions
to $\Omega_f$ and $\Omega_p$ are denoted by
$\mathcal T_h^f$ and $\mathcal T_h^p$, respectively.

For an integer $k\ge 2$, define
\begin{align*}
\bm V_{f,h}
&:=
\left\{
\bm v_h\in\bm V_f\cap\bm C^0(\overline{\Omega_f}):
\bm v_h|_K\in \bm P_{k}(K),\
K\in\mathcal T_h^f
\right\},\\
Q_{f,h}
&:=
\left\{
q_h\in Q_f\cap C^0(\overline{\Omega_f}):
q_h|_K\in P_{k-1}(K),\
K\in\mathcal T_h^f
\right\},\\
\bm V_{p,h}
&:=
\left\{
\bm v_h\in\bm V_p\cap\bm C^0(\overline{\Omega_p}):
\bm v_h|_K\in \bm P_{k}(K),\
K\in\mathcal T_h^p
\right\},\\
Q_{p,h}
&:=
\left\{
q_h\in Q_p\cap C^0(\overline{\Omega_p}):
q_h|_K\in P_{k-1}(K),\
K\in\mathcal T_h^p
\right\},\\
W_{p,h}
&:=
\left\{
w_h\in W_p\cap C^0(\overline{\Omega_p}):
w_h|_K\in P_{k}(K),\
K\in\mathcal T_h^p
\right\}.
\end{align*}
We comment here that the Taylor--Hood pairs
$(\bm V_{f,h},Q_{f,h})$ and $(\bm V_{p,h},Q_{p,h})$ satisfy
\begin{equation}
\label{eq:infsup_stokes_h}
\inf_{0\ne q_h\in Q_{f,h}}
\sup_{0\ne\bm v_h\in\bm V_{f,h}}
\frac{-(\nabla\!\cdot\!\bm v_h,q_h)_{\Omf}}
{\norm{\bm v_h}{1,\Omf}\norm{q_h}{0,\Omf}}
\ge\beta_f,
\end{equation}
and
\begin{equation}
\label{eq:infsup_biot_h}
\inf_{0\ne q_h\in Q_{p,h}}
\sup_{0\ne\bm v_h\in\bm V_{p,h}}
\frac{-(\nabla\!\cdot\!\bm v_h,q_h)_{\Omp}}
{\norm{\bm v_h}{1,\Omp}\norm{q_h}{0,\Omp}}
\ge\beta_p,
\end{equation}
where $\beta_f,\beta_p>0$ are independent of $h$. Set
\[
\bm H_h
:=
\bm V_{f,h}\times\bm V_{p,h}
\times Q_{f,h}\times Q_{p,h}\times W_{p,h}.
\]

For each $n\ge1$, given the discrete solutions at the two preceding
time levels, find
\[
(\bm u_{f,h}^{n+1},\bm\eta_h^{n+1},
p_{f,h}^{n+1},\xi_h^{n+1},p_{p,h}^{n+1})
\in\bm H_h
\]
such that, for every
$(\bm v_{f,h},\bm v_{p,h},w_{f,h},w_{s,h},w_{p,h})\in\bm H_h$,
\begin{subequations}
\label{eq:BDF2_FE_scheme}
\begin{align}
&\rho_f(\Dt{\bm u_{f,h}^{n+1}},\bm v_{f,h})_{\Omf}
+2\mu_f(\bm D(\bm u_{f,h}^{n+1}),
        \bm D(\bm v_{f,h}))_{\Omf}
-(p_{f,h}^{n+1},\nabla\!\cdot\!\bm v_{f,h})_{\Omf}
\notag\\
&\quad
+\frac{\gamma\mu_f}{\sqrt\kappa}
\langle
\Tt(\bm u_{f,h}^{n+1}-\Dt{\bm\eta_h^{n+1}}),
\Tt\bm v_{f,h}
\rangle_{\Gfp}
+\langle p_{p,h}^{n+1},\Tn\bm v_{f,h}\rangle_{\Gfp}
=(\bm f_f^{n+1},\bm v_{f,h})_{\Omf},
\label{eq:FE_a}\\[2mm]
&2\mu_p(\bm D(\bm\eta_h^{n+1}),
        \bm D(\bm v_{p,h}))_{\Omp}
-(\xi_h^{n+1},\nabla\!\cdot\!\bm v_{p,h})_{\Omp}
-\langle p_{p,h}^{n+1},\Tn\bm v_{p,h}\rangle_{\Gfp}
\notag\\
&\quad
+\frac{\gamma\mu_f}{\sqrt\kappa}
\langle
\Tt(\Dt{\bm\eta_h^{n+1}}-\bm u_{f,h}^{n+1}),
\Tt\bm v_{p,h}
\rangle_{\Gfp}
=(\bm f_s^{n+1},\bm v_{p,h})_{\Omp},
\label{eq:FE_b}\\[2mm]
&-(\nabla\!\cdot\!\bm u_{f,h}^{n+1},w_{f,h})_{\Omf}
=(S_f^{n+1},w_{f,h})_{\Omf},
\label{eq:FE_c}\\[2mm]
&-(\nabla\!\cdot\!\bm\eta_h^{n+1},w_{s,h})_{\Omp}
+\frac1{\lambda_p}
(\alpha p_{p,h}^{n+1}-\xi_h^{n+1},w_{s,h})_{\Omp}
=0,
\label{eq:FE_d}\\[2mm]
&(\Dt{\Theta_h^{n+1}},w_{p,h})_{\Omp}
-\langle
\Tn(\bm u_{f,h}^{n+1}-\Dt{\bm\eta_h^{n+1}}),
w_{p,h}
\rangle_{\Gfp}
\notag\\
&\quad
+\frac{\kappa}{\mu_f}
(\nabla p_{p,h}^{n+1},\nabla w_{p,h})_{\Omp}
=(S_p^{n+1},w_{p,h})_{\Omp},
\label{eq:FE_e}
\end{align}
\end{subequations}
where
\[
\Theta_h^{n+1}
:=
c_p p_{p,h}^{n+1}
-\frac{\alpha}{\lambda_p}\xi_h^{n+1}.
\]
This is the conforming Galerkin discretization of
\eqref{eq:BDF2_monolithic_weak}.

\subsection{Convergence analysis}
\label{subsec:convergence}
In this section we derive convergence estimates
for the fully discrete monolithic BDF2 finite
element approximation. The analysis is based on
the standard decomposition of the total error
into projection and discrete components, combined
with consistency estimates and discrete energy
arguments.


The use of independent elliptic and $L^2$ projections leaves
uncontrolled pressure residuals in the discrete energy identity.
We therefore employ a coupled projection that preserves the two
algebraic constraints.
We denote
\[
\begin{aligned}
	a_f(\bm u,\bm v)
	&:=
	2\mu_f(\bm D(\bm u),\bm D(\bm v))_{\Omf}
	+\frac{\gamma\mu_f}{\sqrt\kappa}
	\langle\Tt\bm u,\Tt\bm v\rangle_{\Gfp},\\
	a_p(\bm\eta,\bm v)
	&:=
	2\mu_p(\bm D(\bm\eta),\bm D(\bm v))_{\Omp},\\
	d_p(p,w)
	&:=
	\frac{\kappa}{\mu_f}(\nabla p,\nabla w)_{\Omp}.
\end{aligned}
\]

At each $t\in[0,T]$, define
\[
\Pi_h(\bm u_f,\bm\eta,p_f,\xi,p_p)
=
(\widehat{\bm u}_{f,h},\widehat{\bm\eta}_h,
\widehat p_{f,h},\widehat\xi_h,\widehat p_{p,h})
\in\bm H_h
\]
as follows. First, $\widehat p_{p,h}\in W_{p,h}$ is the Darcy
projection:
\begin{equation}
	\label{eq:coupled_proj_pp}
	d_p(p_p-\widehat p_{p,h},w_{p,h})=0
	\qquad
	\forall w_{p,h}\in W_{p,h}.
\end{equation}
The fluid variables satisfy
\begin{align}
	a_f(\bm u_f-\widehat{\bm u}_{f,h},\bm v_{f,h})
	-(p_f-\widehat p_{f,h},\nabla\!\cdot\!\bm v_{f,h})_{\Omf}
	+\langle p_p-\widehat p_{p,h},\Tn\bm v_{f,h}\rangle_{\Gfp}
	&=0,
	\label{eq:coupled_proj_fluid_a}\\
	(\nabla\!\cdot(\bm u_f-\widehat{\bm u}_{f,h}),q_{f,h})_{\Omf}
	&=0,
	\label{eq:coupled_proj_fluid_b}
\end{align}
for every $(\bm v_{f,h},q_{f,h})\in\bm V_{f,h}\times Q_{f,h}$.
The poroelastic variables satisfy
\begin{align}
	a_p(\bm\eta-\widehat{\bm\eta}_h,\bm v_{p,h})
	-(\xi-\widehat\xi_h,\nabla\!\cdot\!\bm v_{p,h})_{\Omp}
	-\langle p_p-\widehat p_{p,h},\Tn\bm v_{p,h}\rangle_{\Gfp}
	&=0,
	\label{eq:coupled_proj_biot_a}\\
	-(\nabla\!\cdot(\bm\eta-\widehat{\bm\eta}_h),w_{s,h})_{\Omp}
	+\frac1{\lambda_p}
	\bigl(\alpha(p_p-\widehat p_{p,h})
	-(\xi-\widehat\xi_h),w_{s,h}\bigr)_{\Omp}
	&=0
	\label{eq:coupled_proj_biot_b}
\end{align}
for every $(\bm v_{p,h},w_{s,h})\in\bm V_{p,h}\times Q_{p,h}$.

Under some regularity assumptions, we have the following approximation properties for the coupled projection. 

\begin{proposition}[Approximation properties of the coupled projection]
	\label{prop:coupled_projection}
	Assume the discrete inf--sup conditions
	\eqref{eq:infsup_stokes_h}--\eqref{eq:infsup_biot_h}.
	Let
	\[
	(\widehat{\bm u}_{f,h},\widehat{\bm\eta}_h,
	\widehat p_{f,h},\widehat\xi_h,\widehat p_{p,h})
	=
	\Pi_h(\bm u_f,\bm\eta,p_f,\xi,p_p)
	\]
	be defined by
	\eqref{eq:coupled_proj_pp}--\eqref{eq:coupled_proj_biot_b}.
	Then
	\begin{align}
		&\norm{\bm u_f-\widehat{\bm u}_{f,h}}{1,\Omf}
		+\norm{p_f-\widehat p_{f,h}}{0,\Omf}
		+\norm{\bm\eta-\widehat{\bm\eta}_h}{1,\Omp}
		\notag\\
		&\qquad
		+\norm{\xi-\widehat\xi_h}{0,\Omp}
		+\norm{p_p-\widehat p_{p,h}}{1,\Omp}
		\notag\\
		&\quad\le
		Ch^k
		\Bigl(
		\norm{\bm u_f}{k+1,\Omf}
		+\norm{p_f}{k,\Omf}
		+\norm{\bm\eta}{k+1,\Omp}
		+\norm{\xi}{k,\Omp}
		+\norm{p_p}{k+1,\Omp}
		\Bigr).
		\label{eq:coupled_projection_estimate}
	\end{align}
	The constant $C$ is independent of $h$. For
	$\alpha\in(0,1]$ and $\lambda_p\ge\lambda_{\min}>0$, it may be
	chosen uniformly as $\lambda_p\to\infty$. It may depend on the
	fixed coefficients $\mu_f,\mu_p,\kappa,\gamma$ and on the domain.
\end{proposition}

\begin{proof}
	The Darcy projection satisfies
	\[
	\norm{p_p-\widehat p_{p,h}}{1,\Omp}
	\le
	C\inf_{w_{p,h}\in W_{p,h}}
	\norm{p_p-w_{p,h}}{1,\Omp}.
	\]
	The trace theorem therefore gives
	\[
	\norm{p_p-\widehat p_{p,h}}{0,\Gfp}
	\le
	C\norm{p_p-\widehat p_{p,h}}{1,\Omp}.
	\]
	Applying the mixed quasi-optimality estimate to
	\eqref{eq:coupled_proj_fluid_a}--%
	\eqref{eq:coupled_proj_fluid_b} yields
	\[
	\begin{aligned}
		&\norm{\bm u_f-\widehat{\bm u}_{f,h}}{1,\Omf}
		+\norm{p_f-\widehat p_{f,h}}{0,\Omf}
		\\
		&\qquad\le C\left[
		\inf_{\bm v_{f,h}}\norm{\bm u_f-\bm v_{f,h}}{1,\Omf}
		+\inf_{q_{f,h}}\norm{p_f-q_{f,h}}{0,\Omf}
		+\norm{p_p-\widehat p_{p,h}}{0,\Gfp}
		\right].
	\end{aligned}
	\]
	Similarly, the mixed stability of
	\eqref{eq:coupled_proj_biot_a}--%
	\eqref{eq:coupled_proj_biot_b} gives
	\[
	\begin{aligned}
		&\norm{\bm\eta-\widehat{\bm\eta}_h}{1,\Omp}
		+\norm{\xi-\widehat\xi_h}{0,\Omp}
		\\
		&\qquad\le C\left[
		\inf_{\bm v_{p,h}}\norm{\bm\eta-\bm v_{p,h}}{1,\Omp}
		+\inf_{w_{s,h}}\norm{\xi-w_{s,h}}{0,\Omp}
		+\norm{p_p-\widehat p_{p,h}}{0,\Gfp}
		+\frac{\alpha}{\lambda_p}
		\norm{p_p-\widehat p_{p,h}}{0,\Omp}
		\right].
	\end{aligned}
	\]
	The result follows from the finite element approximation
	properties  \cite{BoffiBrezziFortin2013}. Since the projection is linear and time independent,
	it commutes with time differentiation.
\end{proof}

The same estimate holds for the time derivatives required below.
The constant in \eqref{eq:coupled_projection_estimate} is independent
of $h$; it may depend on the fixed coefficients
$\mu_f,\mu_p,\kappa,\gamma$ and on the regularity norms of the exact solution.




\begin{assumption}[Regularity assumption]
	\label{ass:regularity}
	The exact solution is sufficiently regular so that
	\begin{align*}
		\bm u_f
		&\in W^{3,\infty}(0,T;\bm L^2(\Omf))
		\cap W^{1,\infty}(0,T;\bm H^{k+1}(\Omf))
		\cap L^\infty(0,T;\bm H^{k+2}(\Omf)),\\
		\bm\eta
		&\in W^{4,\infty}(0,T;\bm H^1(\Omp))
		\cap W^{2,\infty}(0,T;\bm H^{k+1}(\Omp))
		\cap L^\infty(0,T;\bm H^{k+2}(\Omp)),\\
		p_f
		&\in W^{1,\infty}(0,T;H^k(\Omf)),\\
		\xi
		&\in W^{3,\infty}(0,T;L^2(\Omp))
		\cap W^{1,\infty}(0,T;H^k(\Omp)),\\
		p_p
		&\in W^{3,\infty}(0,T;L^2(\Omp))
		\cap W^{1,\infty}(0,T;H^{k+1}(\Omp)).
	\end{align*}
\end{assumption}

For a function $g\in W^{3,\infty}(0,T;X)$, define
\[
\tau_g^{n+1}
:=
\partial_tg(t^{n+1})-D_tg^{n+1}.
\]
The BDF2 consistency estimate gives
\begin{equation}
	\label{eq:BDF2_consistency}
	\norm{\bm\tau_{\bm u}^{n+1}}{0,\Omf}
	+\norm{\bm\tau_{\bm\eta}^{n+1}}{1,\Omp}
	+\norm{\tau_\Theta^{n+1}}{0,\Omp}
	\le C\Delta t^2,
\end{equation}
where
\[
\bm\tau_{\bm u}^{n+1}
:=\partial_t\bm u_f(t^{n+1})-D_t\bm u_f^{n+1},
\quad
\bm\tau_{\bm\eta}^{n+1}
:=\partial_t\bm\eta(t^{n+1})-D_t\bm\eta^{n+1},
\]
and
\[
\tau_\Theta^{n+1}
:=\partial_t\Theta(t^{n+1})-D_t\Theta^{n+1}.
\]
Such an estimate follows from a Taylor expansion
of the BDF2 difference operator about
$t^{n+1}$; see, for example, \cite{HairerWanner1996}.

\subsection{Error decomposition and error equations}

For each variable $r$ (one component of $(\bm{u}_f, \bm{\eta}, p_f, \xi, p_p)$), write
\[
r^n-r_h^n
=
\underbrace{r^n-\widehat r_h^n}_{\rho_r^n}
+
\underbrace{\widehat r_h^n-r_h^n}_{\theta_r^n}.
\]
Set
\[
\rho_Y^n:=\alpha\rho_{p_p}^n-\rho_\xi^n,
\qquad
\theta_Y^n:=\alpha\theta_{p_p}^n-\theta_\xi^n,
\]
and
\[
\rho_\Theta^n
:=c_p\rho_{p_p}^n-\frac{\alpha}{\lambda_p}\rho_\xi^n,
\qquad
\theta_\Theta^n
:=c_p\theta_{p_p}^n-\frac{\alpha}{\lambda_p}\theta_\xi^n.
\]

We initialize the spatially discrete problem by
$\bm U_h^0=\Pi_h\bm U(0)$, so that $\theta_r^0=0$ for every
component. The value $\bm U_h^1$ is produced by a constraint-compatible
startup method; its error is retained explicitly through
$E_\theta^1$ below.

Subtracting \eqref{eq:BDF2_FE_scheme} from the exact equations at
$t^{n+1}$ and using
\eqref{eq:coupled_proj_pp}--\eqref{eq:coupled_proj_biot_b}, we obtain
\begin{subequations}
	\label{eq:error_system}
	\begin{align}
		&\rho_f(D_t\theta_{\bm u}^{n+1},\bm v_{f,h})_{\Omf}
		+a_f(\theta_{\bm u}^{n+1},\bm v_{f,h})
		-(\theta_{p_f}^{n+1},\nabla\!\cdot\!\bm v_{f,h})_{\Omf}
		\notag\\
		&\quad
		-\frac{\gamma\mu_f}{\sqrt\kappa}
		\langle\Tt D_t\theta_{\bm\eta}^{n+1},
		\Tt\bm v_{f,h}\rangle_{\Gfp}
		+\langle\theta_{p_p}^{n+1},\Tn\bm v_{f,h}\rangle_{\Gfp}
		=\mathcal F_{\bm u}^{n+1}(\bm v_{f,h}),
		\label{eq:error_a}\\
		&a_p(\theta_{\bm\eta}^{n+1},\bm v_{p,h})
		-(\theta_\xi^{n+1},\nabla\!\cdot\!\bm v_{p,h})_{\Omp}
		-\langle\theta_{p_p}^{n+1},\Tn\bm v_{p,h}\rangle_{\Gfp}
		\notag\\
		&\quad
		+\frac{\gamma\mu_f}{\sqrt\kappa}
		\langle
		\Tt(D_t\theta_{\bm\eta}^{n+1}-\theta_{\bm u}^{n+1}),
		\Tt\bm v_{p,h}
		\rangle_{\Gfp}
		=\mathcal F_{\bm\eta}^{n+1}(\bm v_{p,h}),
		\label{eq:error_b}\\
		&-(\nabla\!\cdot\!\theta_{\bm u}^{n+1},w_{f,h})_{\Omf}=0,
		\label{eq:error_c}\\
		&-(\nabla\!\cdot\!\theta_{\bm\eta}^{n+1},w_{s,h})_{\Omp}
		+\frac1{\lambda_p}
		(\alpha\theta_{p_p}^{n+1}-\theta_\xi^{n+1},w_{s,h})_{\Omp}=0,
		\label{eq:error_d}\\
		&(D_t\theta_\Theta^{n+1},w_{p,h})_{\Omp}
		-\langle
		\Tn(\theta_{\bm u}^{n+1}-D_t\theta_{\bm\eta}^{n+1}),
		w_{p,h}
		\rangle_{\Gfp}
		\notag\\
		&\quad
		+\frac{\kappa}{\mu_f}
		(\nabla\theta_{p_p}^{n+1},\nabla w_{p,h})_{\Omp}
		=\mathcal F_{p_p}^{n+1}(w_{p,h}),
		\label{eq:error_e}
	\end{align}
\end{subequations}
where
\begin{align}
	\mathcal F_{\bm u}^{n+1}(\bm v_{f,h})
	&:=
	-\rho_f(D_t\rho_{\bm u}^{n+1}
	+\bm\tau_{\bm u}^{n+1},\bm v_{f,h})_{\Omf}
	\notag\\
	&\quad
	+\frac{\gamma\mu_f}{\sqrt\kappa}
	\langle
	\Tt(D_t\rho_{\bm\eta}^{n+1}
	+\bm\tau_{\bm\eta}^{n+1}),
	\Tt\bm v_{f,h}
	\rangle_{\Gfp},
	\label{eq:F_u}\\
	\mathcal F_{\bm\eta}^{n+1}(\bm v_{p,h})
	&:=
	\frac{\gamma\mu_f}{\sqrt\kappa}
	\langle
	\Tt(\rho_{\bm u}^{n+1}
	-D_t\rho_{\bm\eta}^{n+1}
	-\bm\tau_{\bm\eta}^{n+1}),
	\Tt\bm v_{p,h}
	\rangle_{\Gfp},
	\label{eq:F_eta}\\
	\mathcal F_{p_p}^{n+1}(w_{p,h})
	&:=
	-(D_t\rho_\Theta^{n+1}
	+\tau_\Theta^{n+1},w_{p,h})_{\Omp}
	\notag\\
	&\quad
	+\langle
	\Tn(\rho_{\bm u}^{n+1}
	-D_t\rho_{\bm\eta}^{n+1}
	-\bm\tau_{\bm\eta}^{n+1}),
	w_{p,h}
	\rangle_{\Gfp}.
	\label{eq:F_pp}
\end{align}
The divergence and total-pressure residuals vanish because the coupled projection preserves both algebraic constraints.
We comment here that $\mathcal{F}_{\bm{u}}^{n+1}$ contains no elastic term
from $\rho_{\bm{u}}$, nor does its tangential residual retain
$\rho_{\bm u}^{n+1}$. Likewise, $\mathcal{F}_{\bm{\eta}}^{n+1}$ contains no elastic term from $\rho_{\bm\eta}$. Similarly, the 
term $(\kappa/\mu_f)(\nabla\rho_{p_p},\nabla w_{p,h})$ disappears from $\mathcal{F}_{p_p}^{n+1}$.

\subsection{Error estimates and convergence theorem}
\label{subsec:convergence_theorem}

Define
\begin{align}
	E_\theta^n
	&:=
	\frac{\rho_f}{4}
	\left(
	\norm{\theta_{\bm u}^n}{0,\Omf}^2
	+\norm{2\theta_{\bm u}^n-\theta_{\bm u}^{n-1}}{0,\Omf}^2
	\right)
	+\frac14
	\left(
	\norm{\theta_{\bm\eta}^n}{a_p}^2
	+\norm{2\theta_{\bm\eta}^n-\theta_{\bm\eta}^{n-1}}{a_p}^2
	\right)
	\notag\\
	&\quad
	+\frac{s_0}{4}
	\left(
	\norm{\theta_{p_p}^n}{0,\Omp}^2
	+\norm{2\theta_{p_p}^n-\theta_{p_p}^{n-1}}{0,\Omp}^2
	\right) 
	+\frac1{4\lambda_p}
	\left(
	\norm{\theta_Y^n}{0,\Omp}^2
	+\norm{2\theta_Y^n-\theta_Y^{n-1}}{0,\Omp}^2
	\right),
	\label{eq:def_Etheta}
\end{align}
where
\[
\norm{\bm v}{a_p}^2:=2\mu_p\norm{\bm D(\bm v)}{0,\Omp}^2.
\]
Let $D_\theta^{n+1}$ be the expression in
\eqref{eq:discrete_dissipation} with
$(\bm u_f,\bm\eta,p_p,Y)$ replaced by
$(\theta_{\bm u},\theta_{\bm\eta},\theta_{p_p},\theta_Y)$.

Testing \eqref{eq:error_a}--\eqref{eq:error_e} with
\[
\bm v_{f,h}=4\Delta t\,\theta_{\bm u}^{n+1},
\qquad
\bm v_{p,h}=4\Delta t\,D_t\theta_{\bm\eta}^{n+1},
\qquad
w_{f,h}=-4\Delta t\,\theta_{p_f}^{n+1},
\]
\[
w_{s,h}=0,
\qquad
w_{p,h}=4\Delta t\,\theta_{p_p}^{n+1},
\]
and using the homogeneous constraint \eqref{eq:error_d} at the three
BDF2 levels yields
\[
\nabla\!\cdot D_t\theta_{\bm\eta}^{n+1}
=
\frac1{\lambda_p}D_t\theta_Y^{n+1}.
\]
The same cancellations as in Theorem~\ref{thm:BDF2_energy_stability}
then give
\begin{equation}
	\label{eq:error_energy_identity}
	E_\theta^{n+1}-E_\theta^n
	+\Delta t\,D_\theta^{n+1}
	=
	\Delta t\,
	\mathcal R_\theta^{n+1},
\end{equation}
where
\[
\mathcal R_\theta^{n+1}
=
\mathcal F_{\bm u}^{n+1}(\theta_{\bm u}^{n+1})
+\mathcal F_{\bm\eta}^{n+1}
(D_t\theta_{\bm\eta}^{n+1})
+\mathcal F_{p_p}^{n+1}(\theta_{p_p}^{n+1}).
\]

Set
\[
\mathscr{A}_{h,\Delta t}:=\Delta t^2+h^k.
\]
Using \eqref{eq:coupled_projection_estimate},
\eqref{eq:BDF2_consistency}, the trace, Korn, Poincar\'e, and
Young inequalities, the residual terms not involving
$D_t\theta_{\bm\eta}^{n+1}$ satisfy
\begin{equation}
	\label{eq:nontime_residual_bound}
	\left|
	\mathcal F_{\bm u}^{n+1}(\theta_{\bm u}^{n+1})
	+
	\mathcal F_{p_p}^{n+1}(\theta_{p_p}^{n+1})
	\right|
	\le
	\frac12D_\theta^{n+1}+C \mathscr{A}_{h,\Delta t}^2,
\end{equation}
where $C$ is independent of $h$ and $\Delta t$. It remains to
estimate the term
$\mathcal F_{\bm\eta}^{n+1}(D_t\theta_{\bm\eta}^{n+1})$.

\begin{lemma}[BDF2 summation by parts]
	\label{lem:sbp_time}
	Let $\{G^{n+1}\}_{n\ge1}$ be bounded linear functionals and let
	$z^0=0$. Then
	\begin{align}
		\Delta t\sum_{n=1}^{m}G^{n+1}(D_tz^{n+1})
		&=
		\frac32G^{m+1}(z^{m+1})
		-\frac12G^{m+1}(z^m)
		-\frac32G^2(z^1)
		\notag\\
		&\quad
		-\frac32\sum_{n=2}^{m}
		(G^{n+1}-G^n)(z^n)
		+\frac12\sum_{n=1}^{m-1}
		(G^{n+2}-G^{n+1})(z^n).
		\label{eq:sbp_time_identity}
	\end{align}
\end{lemma}

\begin{proof}
	Using
	\[
	D_tz^{n+1}
	=
	\frac32\delta_tz^{n+1}
	-\frac12\delta_tz^n,
	\]
	the result follows by applying Abel summation to the two
	backward-difference sums.
\end{proof}

\begin{lemma}[Time-differentiated interface residual]
	\label{lem:Feta_summed_bound}
	Let
	\[
	G^{n+1}:=\mathcal F_{\bm\eta}^{n+1}\in\bm V_p',
	\]
	where
	\[
	G^{n+1}(\bm v_{p,h})
	=
	\frac{\gamma\mu_f}{\sqrt{\kappa}}
	\left\langle
	\Tt\left(
	\rho_{\bm u}^{n+1}
	-D_t\rho_{\bm\eta}^{n+1}
	-\bm\tau_{\bm\eta}^{n+1}
	\right),
	\Tt\bm v_{p,h}
	\right\rangle_{\Gfp}.
	\]
	Under Assumption~\ref{ass:regularity} and
	Proposition~\ref{prop:coupled_projection}, for
	$1\le m\le N-1$ and every $\varepsilon>0$,
	\begin{equation}
		\label{eq:Feta_summed_bound}
		\left|
		\Delta t\sum_{n=1}^{m}
		\mathcal F_{\bm\eta}^{n+1}
		(D_t\theta_{\bm\eta}^{n+1})
		\right|
		\le
		\varepsilon
		\max_{1\le j\le m+1}E_\theta^j
		+
		C_\varepsilon\left(\Delta t^4+h^{2k}\right).
	\end{equation}
	Here $C_\varepsilon$ is independent of $h$ and $\Delta t$, but may
	depend on $T$, the fixed physical parameters, the domain constants,
	and the regularity norms of the exact solution.
\end{lemma}

\begin{proof}
	The trace inequality gives
	\begin{align}
		|G^{n+1}(\bm v_{p,h})|
		&\le
		C
		\Bigl(
		\norm{\rho_{\bm u}^{n+1}}{1,\Omf}
		+
		\norm{D_t\rho_{\bm\eta}^{n+1}}{1,\Omp}
		+
		\norm{\bm\tau_{\bm\eta}^{n+1}}{1,\Omp}
		\Bigr)
		\norm{\bm v_{p,h}}{1,\Omp}.
		\label{eq:G_pointwise_pre}
	\end{align}
	Because the coupled projection is linear and time independent,
	Proposition~\ref{prop:coupled_projection} and
	\eqref{eq:BDF2_consistency} imply
	\[
	\norm{\rho_{\bm u}^{n+1}}{1,\Omf}
	+
	\norm{D_t\rho_{\bm\eta}^{n+1}}{1,\Omp}
	\le Ch^k,
	\qquad
	\norm{\bm\tau_{\bm\eta}^{n+1}}{1,\Omp}
	\le C\Delta t^2.
	\]
	Hence,
	\begin{equation}
		\label{eq:G_pointwise}
		\norm{G^{n+1}}{\bm V_p'}
		\le C \mathscr{A}_{h,\Delta t}.
	\end{equation}
	
	For $n\ge2$,
	\begin{align}
		(G^{n+1}-G^n)(\bm v_{p,h})
		&=
		\frac{\gamma\mu_f}{\sqrt{\kappa}}
		\left\langle
		\Tt\left[
		\rho_{\bm u}^{n+1}-\rho_{\bm u}^{n}
		-
		\left(
		D_t\rho_{\bm\eta}^{n+1}
		-D_t\rho_{\bm\eta}^{n}
		\right)
		\right.\right.
		\notag\\
		&\hspace{3.8cm}\left.\left.
		-
		\left(
		\bm\tau_{\bm\eta}^{n+1}
		-\bm\tau_{\bm\eta}^{n}
		\right)
		\right],
		\Tt\bm v_{p,h}
		\right\rangle_{\Gfp}.
		\label{eq:G_difference_identity}
	\end{align}
	The temporal regularity and the projection estimate yield
	\[
	\norm{\rho_{\bm u}^{n+1}-\rho_{\bm u}^{n}}{1,\Omf}
	+
	\norm{
		D_t\rho_{\bm\eta}^{n+1}
		-D_t\rho_{\bm\eta}^{n}
	}{1,\Omp}
	\le
	C\Delta t\,h^k.
	\]
	Moreover, the fourth time derivative of $\bm\eta$ gives
	\begin{equation}
		\label{eq:tau_eta_difference}
		\norm{
			\bm\tau_{\bm\eta}^{n+1}
			-\bm\tau_{\bm\eta}^{n}
		}{1,\Omp}
		\le
		C\Delta t^3
		\max_{t\in[t^{n-2},t^{n+1}]}
		\norm{\partial_t^4\bm\eta(t)}{1,\Omp}.
	\end{equation}
	Consequently,
	\begin{equation}
		\label{eq:G_difference}
		\norm{G^{n+1}-G^n}{\bm V_p'}
		\le
		C\left(\Delta t\,h^k+\Delta t^3\right)
		=
		C\Delta t \mathscr{A}_{h,\Delta t}.
	\end{equation}
	
	Apply Lemma~\ref{lem:sbp_time} with
	$z^n=\theta_{\bm\eta}^n$. Since
	$\theta_{\bm\eta}^0=\bm0$,
	\begin{align}
		&\Delta t\sum_{n=1}^{m}
		G^{n+1}(D_t\theta_{\bm\eta}^{n+1})
		\notag\\
		&\quad=
		\frac32G^{m+1}(\theta_{\bm\eta}^{m+1})
		-\frac12G^{m+1}(\theta_{\bm\eta}^{m})
		-\frac32G^2(\theta_{\bm\eta}^{1})
		\notag\\
		&\qquad
		-\frac32\sum_{n=2}^{m}
		(G^{n+1}-G^n)(\theta_{\bm\eta}^{n})
		+\frac12\sum_{n=1}^{m-1}
		(G^{n+2}-G^{n+1})(\theta_{\bm\eta}^{n}).
		\label{eq:Feta_after_sbp}
	\end{align}
	
	Set
	\[
	M_m:=\max_{1\le j\le m+1}E_\theta^j.
	\]
	From the definition of $E_\theta^j$ and Korn's inequality,
	\begin{equation}
		\label{eq:theta_eta_energy_control}
		\norm{\theta_{\bm\eta}^{j}}{1,\Omp}
		\le
		C\sqrt{E_\theta^j}
		\le
		C\sqrt{M_m},
		\qquad
		1\le j\le m+1.
	\end{equation}
	Therefore, \eqref{eq:G_pointwise} gives the endpoint estimate
	\[
	\left|
	\frac32G^{m+1}(\theta_{\bm\eta}^{m+1})
	-\frac12G^{m+1}(\theta_{\bm\eta}^{m})
	-\frac32G^2(\theta_{\bm\eta}^{1})
	\right|
	\le
	C \mathscr{A}_{h,\Delta t}\sqrt{M_m}.
	\]
	Similarly, by \eqref{eq:G_difference},
	\eqref{eq:theta_eta_energy_control}, and $m\Delta t\le T$,
	\[
	\sum_{n=2}^{m}
	\left|
	(G^{n+1}-G^n)(\theta_{\bm\eta}^{n})
	\right|
	+
	\sum_{n=1}^{m-1}
	\left|
	(G^{n+2}-G^{n+1})(\theta_{\bm\eta}^{n})
	\right|
	\le
	C \mathscr{A}_{h,\Delta t}\sqrt{M_m}.
	\]
	Substitution into \eqref{eq:Feta_after_sbp} yields
	\[
	\left|
	\Delta t\sum_{n=1}^{m}
	\mathcal F_{\bm\eta}^{n+1}
	(D_t\theta_{\bm\eta}^{n+1})
	\right|
	\le
	C \mathscr{A}_{h,\Delta t}\sqrt{M_m}.
	\]
	Young's inequality and
	\[
	\mathscr{A}_{h,\Delta t}^2=(\Delta t^2+h^k)^2
	\le
	2\Delta t^4+2h^{2k}
	\]
	then prove \eqref{eq:Feta_summed_bound}.
\end{proof}

Combining
\eqref{eq:error_energy_identity},
\eqref{eq:nontime_residual_bound}, and
\eqref{eq:Feta_summed_bound}, and then choosing the index at which
$\max_{1\le j\le m+1}E_\theta^j$ is attained, yields the following
estimate.
\begin{theorem}[Discrete error estimate]
	\label{thm:discrete_error_estimate}
	Under Assumption~\ref{ass:regularity}, for $1\le m\le N-1$,
	\begin{equation}
		\label{eq:error_cumulative}
		\max_{1\le j\le m+1}E_\theta^j
		+\Delta t\sum_{n=1}^{m}D_\theta^{n+1}
		\le
		C\left(E_\theta^1+\Delta t^4+h^{2k}\right).
	\end{equation}
	The constant $C$ is independent of $h$ and $\Delta t$.
\end{theorem}

\begin{proof}
	After summing \eqref{eq:error_energy_identity} from $n=1$ to $m$,
	\eqref{eq:nontime_residual_bound} and
	\eqref{eq:Feta_summed_bound} give
	\[
	E_\theta^{m+1}-E_\theta^1
	+\frac{\Delta t}{2}\sum_{n=1}^{m}D_\theta^{n+1}
	\le
	\varepsilon\max_{1\le j\le m+1}E_\theta^j
	+C_\varepsilon \mathscr{A}_{h,\Delta t}^2.
	\]
	Applying this inequality first at an index where the maximum energy is
	attained and choosing $\varepsilon$ sufficiently small gives
	\[
	\max_{1\le j\le m+1}E_\theta^j
	\le C(E_\theta^1+\mathscr{A}_{h,\Delta t}^2).
	\]
	Substitution into the preceding inequality proves
	\eqref{eq:error_cumulative}.
\end{proof}

We are now in a position to state the main
convergence result for the fully-discrete scheme.

\begin{theorem}[Convergence of the fully-discrete scheme]
	\label{thm:BDF2_convergence}
	Assume \eqref{eq:infsup_stokes_h}--\eqref{eq:infsup_biot_h},
	Assumption~\ref{ass:regularity}, and
	$\bm U_h^0=\Pi_h\bm U(0)$. Then, for $1\le n\le N-1$,
	\begin{align}
		&\norm{\bm u_f(t^{n+1})-\bm u_{f,h}^{n+1}}{0,\Omf}
		+\norm{\bm\eta(t^{n+1})-\bm\eta_h^{n+1}}{1,\Omp}
		\notag\\
		&\quad
		+\sqrt{s_0}\,
		\norm{p_p(t^{n+1})-p_{p,h}^{n+1}}{0,\Omp}
		+\frac1{\sqrt{\lambda_p}}
		\norm{Y(t^{n+1})-Y_h^{n+1}}{0,\Omp}
		\notag\\
		&\qquad
		\le
		C\left(\sqrt{E_\theta^1}+\Delta t^2+h^k\right),
		\label{eq:BDF2_convergence_final}
	\end{align}
	and
	\begin{equation}
		\label{eq:BDF2_convergence_pp_H1}
		\left(\frac{\kappa}{\mu_f}\right)^{1/2}
		\left(
		\Delta t\sum_{j=1}^{n}
		\norm{\nabla(p_p(t^{j+1})-p_{p,h}^{j+1})}{0,\Omp}^2
		\right)^{1/2}
		\le
		C\left(\sqrt{E_\theta^1}+\Delta t^2+h^k\right).
	\end{equation}
	The constant is independent of $h$ and $\Delta t$. For
	$0\le s_0 \le \overline s_0$, $\alpha\in(0,1]$, and
	$\lambda_p \ge \underline\lambda>0$, it is uniform as
	$s_0\to0$ and $\lambda_p\to\infty$ in the weighted norms displayed
	above. It may depend on the fixed coefficients
	$\mu_f,\mu_p,\kappa,\gamma$ and on the regularity norms.
\end{theorem}

\begin{proof}
	The result follows from the decomposition
	$r^{n+1}-r_h^{n+1}=\rho_r^{n+1}+\theta_r^{n+1}$,
	the projection estimate \eqref{eq:coupled_projection_estimate},
	Theorem~\ref{thm:discrete_error_estimate}, and Korn's inequality.
	The accumulated pore-pressure gradient estimate follows from the
	Darcy term in $D_\theta^{n+1}$ and the corresponding projection
	estimate.
\end{proof}

\begin{corollary}[Second-order startup]
	\label{cor:second_order_startup}
	If the startup method satisfies
	\[
	E_\theta^1\le C(\Delta t^4+h^{2k}),
	\]
	then
	\[
	\begin{aligned}
		&\norm{\bm u_f(t^{n+1})-\bm u_{f,h}^{n+1}}{0,\Omf}
		+\norm{\bm\eta(t^{n+1})-\bm\eta_h^{n+1}}{1,\Omp}\\
		&\quad
		+\sqrt{s_0}\,
		\norm{p_p(t^{n+1})-p_{p,h}^{n+1}}{0,\Omp}
		+\frac1{\sqrt{\lambda_p}}
		\norm{Y(t^{n+1})-Y_h^{n+1}}{0,\Omp}
		\le C(\Delta t^2+h^k),
	\end{aligned}
	\]
	and the same rate holds in the accumulated Darcy norm
	\eqref{eq:BDF2_convergence_pp_H1}.
\end{corollary}

The estimate above controls the combined pressure
$Y=\alpha p_p-\xi$. A separate pointwise estimate for the fluid
pressure $p_f$, or a parameter-uniform pointwise estimate for $\xi$, requires additional discrete time-derivative and inf--sup estimates; such estimates are not asserted here.

\begin{remark}[Parameter dependence]
	\label{rem:param_robust}
	The estimates
	\eqref{eq:BDF2_convergence_final}--%
	\eqref{eq:BDF2_convergence_pp_H1}
	are formulated in the parameter-weighted quantities
	\[
	\sqrt{s_0}\,\|p_p-p_{p,h}\|_{0,\Omega_p},
	\qquad
	\lambda_p^{-1/2}\|Y-Y_h\|_{0,\Omega_p},
	\qquad
	\left(\frac{\kappa}{\mu_f}\right)^{1/2}
	\left(
	\Delta t\sum_{j=1}^{n}
	\|\nabla(p_p-p_{p,h})^{j+1}\|_{0,\Omega_p}^{2}
	\right)^{1/2},
	\]
	where \(Y=\alpha p_p-\xi\). For
	$0\le s_0\le \overline s_0,
	\alpha\in(0,1], \lambda_p\ge\underline\lambda>0$,
	the constants may be chosen uniformly as \(s_0\to0\) and
	\(\lambda_p\to\infty\), provided the regularity norms and discrete
	inf--sup constants remain uniform.
	
	These estimates do not imply parameter-uniform bounds for
	\(p_p\) or \(\xi\) in unweighted norms. Indeed, an unweighted
	estimate for \(p_p\) generally introduces a factor \(s_0^{-1/2}\).
	Moreover, since
	\[
	\xi=\alpha p_p-Y,
	\]
	a separate estimate for \(\xi\) inherits the lack of uniform control
	of \(p_p\), together with the factor \(\lambda_p^{1/2}\) arising
	from the weighted \(Y\)-estimate. Additional inf--sup or
	time-derivative estimates would be required to obtain sharper
	pointwise bounds for the pressure variables.
	
	No uniform robustness with respect to the singular limit
	\(\kappa\to0\) is asserted here. The pore-pressure gradient is
	controlled only in the weighted Darcy norm
	\((\kappa/\mu_f)^{1/2}\|\nabla(p_p-p_{p,h})\|\), and the constants in
	the coupled projection and trace estimates may depend on the fixed
	coefficients \(\kappa\), \(\mu_f\), and \(\gamma\). In particular,
	the Beavers--Joseph--Saffman coefficient
	\(\gamma\mu_f/\sqrt{\kappa}\) becomes singular as \(\kappa\to0\).
\end{remark}

For the startup accuracy, the second-order estimate in
Corollary~\ref{cor:second_order_startup} requires
\[
E_\theta^1\le C(\Delta t^4+h^{2k}).
\]
Without such a startup analysis, the above bound should be retained as an explicit assumption.
A single backward Euler startup can satisfy this bound when it is
initialized from compatible projected data and a uniform one-step
resolvent estimate is available.  

\begin{remark}[Backward Euler as a startup method]
	\label{rem:BE_startup}
	The estimate
	\[
	E_\theta^1\le C\bigl(\Delta t^4+h^{2k}\bigr)
	\]
	does not follow merely from the use of the backward Euler method.
	Backward Euler is globally first-order accurate when applied over a
	fixed time interval, but a single backward Euler step starting from
	exact, or sufficiently accurate and constraint-compatible, initial
	data has a local solution error of order $O(\Delta t^2)$.
	
	For the present Stokes--Biot problem, the system has a
	differential--algebraic structure and contains time derivatives in
	the interface coupling. Consequently, the estimate above additionally
	requires a uniform one-step stability estimate for the discrete
	backward Euler operator
	\[
	\frac{1}{\Delta t}\mathcal N_h+\mathcal M_h.
	\]
	Thus, the second-order startup estimate can be proved under
	such a one-step resolvent bound or retained as an explicit assumption. 
\end{remark}

\section{Numerical experiments}
\label{sec:numerical}

To verify the temporal convergence of the proposed monolithic BDF2
scheme, we consider a manufactured solution in a coupled
fluid--poroelastic configuration. The computational domains are
\[
\Omega_f = (0,1)\times(0,1),
\qquad
\Omega_p = (0,1)\times(-1,0),
\qquad
\Gamma_{fp} = (0,1)\times\{0\},
\]
with outward unit normal vectors and tangential direction on
$\Gamma_{fp}$ given by
\[
\bm{n}_f =
\begin{pmatrix} 0 \\ -1 \end{pmatrix},
\qquad
\bm{n}_p =
\begin{pmatrix} 0 \\ 1 \end{pmatrix},
\qquad
\bm{\tau} =
\begin{pmatrix} 1 \\ 0 \end{pmatrix}.
\]
This geometry is convenient for implementation in \texttt{FreeFEM++}
since the interface is flat and mesh-aligned. The same manufactured
solution was used in \cite{GuoLi2026}, allowing direct comparison
with the decoupled scheme studied there.

\medskip
\noindent
The exact solution is prescribed as
\begin{align*}
	\bm{u}_f(x,y,t)
	&= \pi\cos(\pi t)
	\begin{pmatrix}
		-3x + \cos(y) \\ y + 1
	\end{pmatrix}, \qquad
	p_f(x,y,t)
	= e^{t}\sin(\pi x)\cos\!\Bigl(\tfrac{\pi y}{2}\Bigr)
	+ 2\pi\cos(\pi t), \\
	\bm{\eta}(x,y,t)
	&= \sin(\pi t)
	\begin{pmatrix}
		-3x + \cos(y) \\ y + 1
	\end{pmatrix}, \qquad
	p_p(x,y,t)
	= e^{t}\sin(\pi x)\cos\!\Bigl(\tfrac{\pi y}{2}\Bigr),
\end{align*}
and the total pressure is recovered from the definition
\eqref{eq:xi_definition}.
The forcing and source terms, $\bm{f}_f$, $\bm{f}_s$, $S_f$, and $S_p$, together with the exterior boundary data and any nonhomogeneous interface data, are obtained by substituting the exact solution into the strong formulation of the coupled Stokes--Biot system. In the finite element implementation, the nonhomogeneous Dirichlet boundary conditions are imposed using the standard lifting technique.

The physical parameters are
\[
\mu_f = 1, \quad \rho_f = 1, \quad
\mu_p = 1, \quad \lambda_p = 1, \quad
\alpha = 1, \quad s_0 = 0.1, \quad
\kappa = 1, \quad \gamma = 0.
\]
The choice $\gamma = 0$ deactivates the tangential
Beavers--Joseph--Saffman (BJS) friction condition, resulting in a
frictionless interface. This allows us to isolate the normal
coupling mechanism and assess the temporal accuracy of the scheme
without additional interface dissipation effects. Note that the
stability and convergence analysis in
Sections~\ref{sec:analysis}--\ref{sec:convergence} remains valid for
all $\gamma \ge 0$.

\subsection{Temporal convergence}
\label{subsec:temporal}
To assess the temporal accuracy of the proposed monolithic
BDF2 scheme, 
the spatial discretizations use $k=2$, namely, $\bm{P}_2 /P_1$ Taylor–Hood pairs and $P_2$ elements for pore pressure. 
A small mesh size  
$h \approx 1/128$ (39{,}280 triangles and
19{,}897 vertices) is taken so that spatial discretization errors are negligible relative to temporal errors over the range of time
steps considered. The final time is set to $T=0.5$, and the
time step is successively refined according to $
\Delta t \in
\left\{
0.1,\;
0.05,\;
0.025,\;
0.0125
\right\}.
$
The numerical solution is initialized using the exact solution
at $t^0=0$. The first time level is computed using a Backward
Euler step, after which the BDF2 discretization is employed
for all subsequent time levels, consistent with the fully
discrete scheme described in
Section~\ref{sec:BDF2}.
Table~\ref{tab:temporal} reports the errors at the final
time $T = 0.5$ and the observed convergence rates
\[
r = \log_2\!\Bigl(\frac{e(\Delta t)}{e(\Delta t/2)}\Bigr),
\]

\begin{table}[h!]
	\centering
	\caption{Temporal convergence of the monolithic BDF2 scheme.
		Fixed mesh $h \approx 1/128$, $T = 0.5$, standard Backward Euler
		startup at $n=1$, BDF2 for all $n \geq 2$.}
	\label{tab:temporal}
	\footnotesize
	\setlength{\tabcolsep}{3pt}
	\renewcommand{\arraystretch}{1.2}
	\resizebox{\textwidth}{!}{%
		\begin{tabular}{c cc cc cc cc cc cc cc cc}
			\toprule
			$\Delta t$
			& $\|e_{\bm{u}_f}\|_{0,\Omega_f}$ & $r$
			& $\|e_{\bm{u}_f}\|_{1,\Omega_f}$ & $r$
			& $\|e_{\bm{\eta}}\|_{0,\Omega_p}$ & $r$
			& $\|e_{\bm{\eta}}\|_{1,\Omega_p}$ & $r$
			& $\|e_{p_f}\|_{0,\Omega_f}$ & $r$
			& $\|e_{\xi}\|_{0,\Omega_p}$ & $r$
			& $\|e_{p_p}\|_{1,\Omega_p}$ & $r$
			& $\|e_{p_p}\|_{0,\Omega_p}$ & $r$ \\
			\midrule
			$0.10000$
			& $1.949\!\times\!10^{-3}$ & $-$
			& $1.320\!\times\!10^{-2}$ & $-$
			& $1.378\!\times\!10^{-3}$ & $-$
			& $5.487\!\times\!10^{-3}$ & $-$
			& $2.662\!\times\!10^{-1}$ & $-$
			& $1.303\!\times\!10^{-2}$ & $-$
			& $4.636\!\times\!10^{-2}$ & $-$
			& $1.220\!\times\!10^{-2}$ & $-$ \\
			\addlinespace
			$0.05000$
			& $4.878\!\times\!10^{-4}$ & $2.0$
			& $3.335\!\times\!10^{-3}$ & $2.0$
			& $2.657\!\times\!10^{-4}$ & $2.4$
			& $1.061\!\times\!10^{-3}$ & $2.4$
			& $6.704\!\times\!10^{-2}$ & $2.0$
			& $2.530\!\times\!10^{-3}$ & $2.4$
			& $9.002\!\times\!10^{-3}$ & $2.4$
			& $2.382\!\times\!10^{-3}$ & $2.4$ \\
			\addlinespace
			$0.02500$
			& $1.237\!\times\!10^{-4}$ & $2.0$
			& $8.476\!\times\!10^{-4}$ & $2.0$
			& $5.659\!\times\!10^{-5}$ & $2.2$
			& $2.264\!\times\!10^{-4}$ & $2.2$
			& $1.672\!\times\!10^{-2}$ & $2.0$
			& $5.418\!\times\!10^{-4}$ & $2.2$
			& $1.930\!\times\!10^{-3}$ & $2.2$
			& $5.120\!\times\!10^{-4}$ & $2.2$ \\
			\addlinespace
			$0.01250$
			& $3.108\!\times\!10^{-5}$ & $2.0$
			& $2.133\!\times\!10^{-4}$ & $2.0$
			& $1.293\!\times\!10^{-5}$ & $2.1$
			& $5.211\!\times\!10^{-5}$ & $2.1$
			& $4.169\!\times\!10^{-3}$ & $2.0$
			& $1.261\!\times\!10^{-4}$ & $2.1$
			& $4.483\!\times\!10^{-4}$ & $2.1$
			& $1.177\!\times\!10^{-4}$ & $2.1$ \\
			\bottomrule
	\end{tabular}}
\end{table}

\noindent
The results confirm second-order convergence in time for all
error quantities, in agreement with
Theorem~\ref{thm:BDF2_convergence}.
The fluid velocity $\bm{u}_f$ and fluid pressure $p_f$
exhibit uniform convergence rates of approximately two
throughout the refinement sequence. The displacement
$\bm{\eta}$, total pressure $\xi$, and pore pressure
$p_p$ display slightly larger rates on the coarser levels,
but these gradually approach the theoretical value of two
as the time step is refined. This behaviour is typical of
the pre-asymptotic regime and indicates convergence toward
the asymptotic range predicted by the analysis.

The use of a Backward Euler startup step does not affect
the observed second-order convergence. Despite the
first-order initialization, all variables recover the
expected temporal accuracy, confirming that the global
error is governed by the BDF2 discretization.

The pore-pressure errors provide additional verification
of the theoretical estimates. In particular, the
$L^2(\Omega_p)$ error of $p_p$ exhibits second-order
convergence, consistent with the pointwise-in-time estimate
\eqref{eq:BDF2_convergence_final}. Likewise, the observed
$H^1(\Omega_p)$ convergence of $p_p$ is in agreement with
the estimate \eqref{eq:BDF2_convergence_pp_H1}. Overall,
the numerical results provide strong confirmation of the
second-order temporal accuracy established in
Theorem~\ref{thm:BDF2_convergence}.

\subsection{Spatial convergence}
\label{subsec:spatial}

\noindent
To assess the spatial accuracy of the proposed monolithic
BDF2 scheme, the time step is fixed at
$\Delta t=10^{-3}$ so that temporal discretization errors
are negligible relative to spatial discretization errors on
all meshes considered. The mesh is successively refined from
$h\approx 1/16$ to $h\approx 1/128$, while the final time is
kept fixed at $T=0.5$. The numerical solution is initialized
using the exact solution at $t^0=0$, the first time level is
computed using a Backward Euler step, and the BDF2 scheme is
applied thereafter.
\medskip
\[
\log_2
\left(
\frac{e(h)}
{e(h/2)}
\right),
\]
where $e(h)$ denotes the corresponding error norm.
Convergence rates are reported only from the second mesh
level onward.

\begin{table}[h!]
	\centering
	\caption{Spatial convergence of the monolithic BDF2
		scheme. Fixed time step $\Delta t = 0.001$,
		$T = 0.5$, $N = 500$.}
	\label{tab:spatial}
	\footnotesize
	\setlength{\tabcolsep}{3pt}
	\renewcommand{\arraystretch}{1.2}
	\resizebox{\textwidth}{!}{
		\begin{tabular}{c cc cc cc cc cc cc cc cc}
			\toprule
			$h$
			& $\|e_{\bm u_f}\|_{0,\Omf}$ & $r$
			& $\|e_{\bm u_f}\|_{1,\Omf}$ & $r$
			& $\|e_{\bm\eta}\|_{0,\Omp}$ & $r$
			& $\|e_{\bm\eta}\|_{1,\Omp}$ & $r$
			& $\|e_{p_f}\|_{0,\Omf}$ & $r$
			& $\|e_{\xi}\|_{0,\Omp}$ & $r$
			& $\|e_{p_p}\|_{1,\Omp}$ & $r$
			& $\|e_{p_p}\|_{0,\Omp}$ & $r$ \\
			\midrule
			$1/16$
			& $4.318\!\times\!10^{-4}$ & $-$
			& $9.948\!\times\!10^{-3}$ & $-$
			& $5.091\!\times\!10^{-4}$ & $-$
			& $1.049\!\times\!10^{-2}$ & $-$
			& $3.323\!\times\!10^{-2}$ & $-$
			& $3.297\!\times\!10^{-2}$ & $-$
			& $8.558\!\times\!10^{-2}$ & $-$
			& $2.741\!\times\!10^{-3}$ & $-$ \\
			\addlinespace
			$1/32$
			& $5.233\!\times\!10^{-5}$ & $3.0$
			& $2.093\!\times\!10^{-3}$ & $2.2$
			& $4.423\!\times\!10^{-5}$ & $3.5$
			& $1.785\!\times\!10^{-3}$ & $2.6$
			& $6.808\!\times\!10^{-3}$ & $2.3$
			& $6.432\!\times\!10^{-3}$ & $2.4$
			& $2.013\!\times\!10^{-2}$ & $2.1$
			& $3.564\!\times\!10^{-4}$ & $2.9$ \\
			\addlinespace
			$1/64$
			& $5.427\!\times\!10^{-6}$ & $3.3$
			& $4.260\!\times\!10^{-4}$ & $2.3$
			& $4.711\!\times\!10^{-6}$ & $3.2$
			& $3.666\!\times\!10^{-4}$ & $2.3$
			& $1.567\!\times\!10^{-3}$ & $2.1$
			& $1.424\!\times\!10^{-3}$ & $2.2$
			& $4.686\!\times\!10^{-3}$ & $2.1$
			& $3.940\!\times\!10^{-5}$ & $3.2$ \\
			\addlinespace
			$1/128$
			& $5.596\!\times\!10^{-7}$ & $3.3$
			& $8.315\!\times\!10^{-5}$ & $2.4$
			& $6.019\!\times\!10^{-7}$ & $3.0$
			& $8.550\!\times\!10^{-5}$ & $2.1$
			& $3.502\!\times\!10^{-4}$ & $2.2$
			& $3.542\!\times\!10^{-4}$ & $2.0$
			& $1.125\!\times\!10^{-3}$ & $2.1$
			& $5.186\!\times\!10^{-6}$ & $2.9$ \\
			\bottomrule
		\end{tabular}
	}
\end{table}
\noindent
The results demonstrate optimal spatial convergence for all
primary variables, in agreement with
Theorem~\ref{thm:BDF2_convergence}. The fluid velocity
$\bm{u}_f$ and solid displacement $\bm{\eta}$ exhibit
approximately second-order convergence in the $H^1$ norm and
higher-order convergence in the $L^2$ norm. The fluid
pressure $p_f$, total pressure $\xi$, and pore pressure
$p_p$ all converge at approximately second order in their
respective norms.

The observed rates remain stable throughout the refinement
sequence and are fully consistent with the theoretical error
estimates. In particular, the $L^2$ errors of the velocity
and displacement decrease at a rate exceeding that observed
in the corresponding $H^1$ norms, reflecting the improved
accuracy commonly observed for these variables on sufficiently
fine meshes.

Since the time step is fixed at $\Delta t=10^{-3}$, the
temporal discretization error remains negligible throughout
the experiment. Consequently, the observed convergence rates
accurately reflect the spatial approximation properties of
the finite element discretization. Overall, the numerical
results provide strong confirmation of the spatial accuracy
predicted by Theorem~\ref{thm:BDF2_convergence}.

\subsection{Parameter robustness tests}
\label{subsec:robustness}
\noindent
To investigate the robustness of the proposed method with
respect to the physical parameters, we repeat the spatial
convergence study under several parameter regimes while
keeping $\Delta t=10^{-3}$ and $T=0.5$ fixed. At this time
step, temporal discretization errors are negligible, so the
observed behavior reflects the spatial approximation
properties of the method.

We first consider the influence of the Lam\'e parameter
$\lambda_p$, which is known to cause locking difficulties in
classical displacement--pressure formulations of Biot's
equations as the nearly incompressible limit is approached.
To assess robustness in this regime, we compare the cases
$\nu_p=0.3$ ($\lambda_p\approx0.577$) and
$\nu_p=0.4999$ ($\lambda_p\approx1666.4$), spanning nearly
three orders of magnitude in $\lambda_p$.

Table~\ref{tab:robust_nu} reports the corresponding errors
and convergence rates.
\begin{table}[H]
	\centering
	\caption{Parameter test for $\lambda_p$:
		$\nu_p \in \{0.3,\, 0.4999\}$,
		$\kappa = 1$, $s_0 = 1$, $\alpha = 1$,
		$\Delta t = 10^{-3}$, $T = 0.5$.}
	\label{tab:robust_nu}
	\footnotesize
	\setlength{\tabcolsep}{3pt}
	\renewcommand{\arraystretch}{1.2}
	\resizebox{\textwidth}{!}{
		\begin{tabular}{c cc cc cc cc cc cc cc cc}
			\toprule
			\multicolumn{17}{c}{$\nu_p = 0.3$,\quad
				$\lambda_p = 0.577$,\quad $\mu_p = 0.385$} \\
			\midrule
			$h$
			& $\|e_{\bm u_f}\|_{0,\Omf}$ & $r$
			& $\|e_{\bm u_f}\|_{1,\Omf}$ & $r$
			& $\|e_{\bm\eta}\|_{0,\Omp}$ & $r$
			& $\|e_{\bm\eta}\|_{1,\Omp}$ & $r$
			& $\|e_{p_f}\|_{0,\Omf}$ & $r$
			& $\|e_{\xi}\|_{0,\Omp}$ & $r$
			& $\|e_{p_p}\|_{1,\Omp}$ & $r$
			& $\|e_{p_p}\|_{0,\Omp}$ & $r$ \\
			\midrule
			$1/16$
			& $4.321\!\times\!10^{-4}$ & $-$
			& $9.949\!\times\!10^{-3}$ & $-$
			& $1.324\!\times\!10^{-3}$ & $-$
			& $2.709\!\times\!10^{-2}$ & $-$
			& $3.323\!\times\!10^{-2}$ & $-$
			& $3.297\!\times\!10^{-2}$ & $-$
			& $8.562\!\times\!10^{-2}$ & $-$
			& $2.682\!\times\!10^{-3}$ & $-$ \\
			\addlinespace
			$1/32$
			& $5.246\!\times\!10^{-5}$ & $3.0$
			& $2.094\!\times\!10^{-3}$ & $2.2$
			& $1.211\!\times\!10^{-4}$ & $3.5$
			& $4.604\!\times\!10^{-3}$ & $2.6$
			& $6.812\!\times\!10^{-3}$ & $2.3$
			& $6.433\!\times\!10^{-3}$ & $2.4$
			& $2.013\!\times\!10^{-2}$ & $2.1$
			& $3.633\!\times\!10^{-4}$ & $2.9$ \\
			\addlinespace
			$1/64$
			& $5.454\!\times\!10^{-6}$ & $3.3$
			& $4.261\!\times\!10^{-4}$ & $2.3$
			& $1.601\!\times\!10^{-5}$ & $2.9$
			& $9.416\!\times\!10^{-4}$ & $2.3$
			& $1.569\!\times\!10^{-3}$ & $2.1$
			& $1.425\!\times\!10^{-3}$ & $2.2$
			& $4.688\!\times\!10^{-3}$ & $2.1$
			& $4.633\!\times\!10^{-5}$ & $3.0$ \\
			\addlinespace
			$1/128$
			& $5.663\!\times\!10^{-7}$ & $3.3$
			& $8.315\!\times\!10^{-5}$ & $2.4$
			& $3.276\!\times\!10^{-6}$ & $2.3$
			& $2.192\!\times\!10^{-4}$ & $2.1$
			& $3.513\!\times\!10^{-4}$ & $2.2$
			& $3.543\!\times\!10^{-4}$ & $2.0$
			& $1.126\!\times\!10^{-3}$ & $2.1$
			& $8.816\!\times\!10^{-6}$ & $2.4$ \\
			\midrule
			\multicolumn{17}{c}{$\nu_p = 0.4999$,\quad
				$\lambda_p = 1666.4$,\quad $\mu_p = 0.333$} \\
			\midrule
			$h$
			& $\|e_{\bm u_f}\|_{0,\Omf}$ & $r$
			& $\|e_{\bm u_f}\|_{1,\Omf}$ & $r$
			& $\|e_{\bm\eta}\|_{0,\Omp}$ & $r$
			& $\|e_{\bm\eta}\|_{1,\Omp}$ & $r$
			& $\|e_{p_f}\|_{0,\Omf}$ & $r$
			& $\|e_{\xi}\|_{0,\Omp}$ & $r$
			& $\|e_{p_p}\|_{1,\Omp}$ & $r$
			& $\|e_{p_p}\|_{0,\Omp}$ & $r$ \\
			\midrule
			$1/16$
			& $4.317\!\times\!10^{-4}$ & $-$
			& $9.948\!\times\!10^{-3}$ & $-$
			& $1.421\!\times\!10^{-3}$ & $-$
			& $3.103\!\times\!10^{-2}$ & $-$
			& $3.323\!\times\!10^{-2}$ & $-$
			& $3.308\!\times\!10^{-2}$ & $-$
			& $8.557\!\times\!10^{-2}$ & $-$
			& $2.765\!\times\!10^{-3}$ & $-$ \\
			\addlinespace
			$1/32$
			& $5.230\!\times\!10^{-5}$ & $3.0$
			& $2.093\!\times\!10^{-3}$ & $2.2$
			& $1.308\!\times\!10^{-4}$ & $3.4$
			& $5.305\!\times\!10^{-3}$ & $2.5$
			& $6.808\!\times\!10^{-3}$ & $2.3$
			& $6.437\!\times\!10^{-3}$ & $2.4$
			& $2.013\!\times\!10^{-2}$ & $2.1$
			& $3.584\!\times\!10^{-4}$ & $2.9$ \\
			\addlinespace
			$1/64$
			& $5.426\!\times\!10^{-6}$ & $3.3$
			& $4.260\!\times\!10^{-4}$ & $2.3$
			& $1.369\!\times\!10^{-5}$ & $3.3$
			& $1.084\!\times\!10^{-3}$ & $2.3$
			& $1.567\!\times\!10^{-3}$ & $2.1$
			& $1.425\!\times\!10^{-3}$ & $2.2$
			& $4.686\!\times\!10^{-3}$ & $2.1$
			& $3.903\!\times\!10^{-5}$ & $3.2$ \\
			\addlinespace
			$1/128$
			& $5.616\!\times\!10^{-7}$ & $3.3$
			& $8.315\!\times\!10^{-5}$ & $2.4$
			& $1.590\!\times\!10^{-6}$ & $3.1$
			& $2.524\!\times\!10^{-4}$ & $2.1$
			& $3.499\!\times\!10^{-4}$ & $2.2$
			& $3.543\!\times\!10^{-4}$ & $2.0$
			& $1.125\!\times\!10^{-3}$ & $2.1$
			& $4.629\!\times\!10^{-6}$ & $3.1$ \\
			\bottomrule
		\end{tabular}
	}
\end{table}
\noindent
The results are nearly identical for both parameter values
and remain fully consistent with the spatial convergence
behavior observed in Table~\ref{tab:spatial}. In particular,
the convergence rates of all variables are essentially
unaffected by the large increase in $\lambda_p$, and no
deterioration of the displacement approximation is observed
in the nearly incompressible regime. These results provide
numerical evidence that the proposed formulation remains
stable and accurate as $\lambda_p$ becomes large, in
agreement with the estimates of
Theorems~\ref{thm:BDF2_energy_stability} and
\ref{thm:BDF2_convergence}.

\noindent
We next investigate the influence of the permeability
coefficient $\kappa$, which controls the rate of pore-fluid
flow through the Darcy term $(\kappa/\mu_f)\nabla p_p$.
To assess the robustness of the method across different
flow regimes, we consider both a moderate permeability
$\kappa=10^{-1}$ and a nearly impermeable case $\kappa=10^{-6}$.
The remaining parameters are fixed at $\nu_p=0.3$, $s_0=1$, and $\alpha=1$ unless stated otherwise.
\begingroup
\setlength{\intextsep}{0pt}
\begin{table}[H]
	\centering
	\caption{Parameter test for $\kappa$:
		$\kappa = 10^{-1}$, $\nu_p = 0.3$,
		$s_0 = 1$, $\alpha = 1$,
		$\Delta t = 10^{-3}$, $T = 0.5$.}
	\label{tab:robust_kappa}
	\footnotesize
	\setlength{\tabcolsep}{3pt}
	\renewcommand{\arraystretch}{1.2}
	\resizebox{\textwidth}{!}{
		\begin{tabular}{c cc cc cc cc cc cc cc cc}
			\toprule
			$h$
			& $\|e_{\bm u_f}\|_{0,\Omf}$ & $r$
			& $\|e_{\bm u_f}\|_{1,\Omf}$ & $r$
			& $\|e_{\bm\eta}\|_{0,\Omp}$ & $r$
			& $\|e_{\bm\eta}\|_{1,\Omp}$ & $r$
			& $\|e_{p_f}\|_{0,\Omf}$ & $r$
			& $\|e_{\xi}\|_{0,\Omp}$ & $r$
			& $\|e_{p_p}\|_{1,\Omp}$ & $r$
			& $\|e_{p_p}\|_{0,\Omp}$ & $r$ \\
			\midrule
			$1/16$
			& $4.819\!\times\!10^{-4}$ & $-$
			& $1.019\!\times\!10^{-2}$ & $-$
			& $2.424\!\times\!10^{-3}$ & $-$
			& $2.989\!\times\!10^{-2}$ & $-$
			& $4.428\!\times\!10^{-2}$ & $-$
			& $4.197\!\times\!10^{-2}$ & $-$
			& $1.395\!\times\!10^{-1}$ & $-$
			& $2.804\!\times\!10^{-2}$ & $-$ \\
			\addlinespace
			$1/32$
			& $7.737\!\times\!10^{-5}$ & $2.6$
			& $2.177\!\times\!10^{-3}$ & $2.2$
			& $7.151\!\times\!10^{-4}$ & $1.8$
			& $5.694\!\times\!10^{-3}$ & $2.4$
			& $1.018\!\times\!10^{-2}$ & $2.1$
			& $8.629\!\times\!10^{-3}$ & $2.3$
			& $3.029\!\times\!10^{-2}$ & $2.2$
			& $6.064\!\times\!10^{-3}$ & $2.2$ \\
			\addlinespace
			$1/64$
			& $1.462\!\times\!10^{-5}$ & $2.4$
			& $4.481\!\times\!10^{-4}$ & $2.3$
			& $1.756\!\times\!10^{-4}$ & $2.0$
			& $1.245\!\times\!10^{-3}$ & $2.2$
			& $2.413\!\times\!10^{-3}$ & $2.1$
			& $1.960\!\times\!10^{-3}$ & $2.1$
			& $6.982\!\times\!10^{-3}$ & $2.1$
			& $1.403\!\times\!10^{-3}$ & $2.1$ \\
			\addlinespace
			$1/128$
			& $3.580\!\times\!10^{-6}$ & $2.0$
			& $9.065\!\times\!10^{-5}$ & $2.3$
			& $4.828\!\times\!10^{-5}$ & $1.9$
			& $3.100\!\times\!10^{-4}$ & $2.0$
			& $6.113\!\times\!10^{-4}$ & $2.0$
			& $4.978\!\times\!10^{-4}$ & $2.0$
			& $1.740\!\times\!10^{-3}$ & $2.0$
			& $3.611\!\times\!10^{-4}$ & $2.0$ \\
			\bottomrule
		\end{tabular}
	}
\end{table}
\endgroup
\noindent
The results show optimal spatial convergence for all
variables throughout the refinement sequence.
In particular, the pore pressure $p_p$ converges at the
expected rate in both $L^2(\Omega_p)$ and
$H^1(\Omega_p)$, and no deterioration relative to the
reference case $\kappa=1$ is observed.
These results indicate that the proposed formulation
remains stable and accurate for moderate permeability
values, with convergence behavior essentially identical
to that observed in the reference configuration.

\begin{remark}[Parameter dependence of the stability estimate]
	\label{rem:param_dependence}
	As discussed in Remark~\ref{rem:param_robust}, the stability estimate
	\eqref{eq:BDF2_energy_estimate_clean} is robust with respect to
	$\lambda_p$ and $s_0$ only in the corresponding parameter-weighted
	energy norms involving the combined pressure
	$Y=\alpha p_p-\xi$. The permeability $\kappa$ enters explicitly through
	the weighted Darcy and Beavers--Joseph--Saffman dissipation terms, so no
	uniform estimate in standard unweighted norms is asserted as
	$\kappa\to0$. The reduced pore-pressure accuracy observed in the
	low-permeability experiments is therefore consistent with the
	parameter-weighted stability theory and does not contradict
	\eqref{eq:BDF2_energy_estimate_clean}. 
\end{remark}

To further challenge the method, we consider the
low-permeability regime $\kappa=10^{-6}$, for which the
Darcy diffusion term becomes extremely small.
In this regime the pore-pressure equation develops a
singularly perturbed character, and sharp pressure
gradients may arise near the boundary.
Table~\ref{tab:robust_kappa01} reports the results for
both $\nu_p=0.3$ and $\nu_p=0.499$.
\begin{table}[H]
	\centering
	\caption{Parameter test for $\kappa$:
		$\kappa = 10^{-6}$, $\nu_p \in \{0.3,\, 0.499\}$,
		$s_0 = 1$, $\alpha = 1$,
		$\Delta t = 10^{-3}$, $T = 0.5$.}
	\label{tab:robust_kappa01}
	\footnotesize
	\setlength{\tabcolsep}{3pt}
	\renewcommand{\arraystretch}{1.2}
	\resizebox{\textwidth}{!}{
		\begin{tabular}{c cc cc cc cc cc cc cc cc}
			\toprule
			\multicolumn{17}{c}{$\nu_p = 0.3$,\quad
				$\lambda_p = 0.577$,\quad $\kappa = 10^{-6}$} \\
			\midrule
			$h$
			& $\|e_{\bm u_f}\|_{0,\Omf}$ & $r$
			& $\|e_{\bm u_f}\|_{1,\Omf}$ & $r$
			& $\|e_{\bm\eta}\|_{0,\Omp}$ & $r$
			& $\|e_{\bm\eta}\|_{1,\Omp}$ & $r$
			& $\|e_{p_f}\|_{0,\Omf}$ & $r$
			& $\|e_{\xi}\|_{0,\Omp}$ & $r$
			& $\|e_{p_p}\|_{1,\Omp}$ & $r$
			& $\|e_{p_p}\|_{0,\Omp}$ & $r$ \\
			\midrule
			$1/16$
			& $4.471\!\times\!10^{-4}$ & $-$
			& $1.001\!\times\!10^{-2}$ & $-$
			& $2.540\!\times\!10^{-3}$ & $-$
			& $3.157\!\times\!10^{-2}$ & $-$
			& $6.310\!\times\!10^{-2}$ & $-$
			& $6.171\!\times\!10^{-2}$ & $-$
			& $4.087\!\times\!10^{-1}$ & $-$
			& $5.590\!\times\!10^{-2}$ & $-$ \\
			\addlinespace
			$1/32$
			& $5.949\!\times\!10^{-5}$ & $2.9$
			& $2.107\!\times\!10^{-3}$ & $2.2$
			& $6.574\!\times\!10^{-4}$ & $2.0$
			& $6.012\!\times\!10^{-3}$ & $2.4$
			& $1.374\!\times\!10^{-2}$ & $2.2$
			& $1.254\!\times\!10^{-2}$ & $2.3$
			& $1.327\!\times\!10^{-1}$ & $1.6$
			& $1.163\!\times\!10^{-2}$ & $2.3$ \\
			\addlinespace
			$1/64$
			& $9.305\!\times\!10^{-6}$ & $2.7$
			& $4.327\!\times\!10^{-4}$ & $2.3$
			& $1.481\!\times\!10^{-4}$ & $2.2$
			& $1.258\!\times\!10^{-3}$ & $2.3$
			& $3.239\!\times\!10^{-3}$ & $2.1$
			& $2.886\!\times\!10^{-3}$ & $2.1$
			& $5.261\!\times\!10^{-2}$ & $1.3$
			& $2.677\!\times\!10^{-3}$ & $2.1$ \\
			\addlinespace
			$1/128$
			& $1.790\!\times\!10^{-6}$ & $2.4$
			& $8.473\!\times\!10^{-5}$ & $2.4$
			& $4.005\!\times\!10^{-5}$ & $1.9$
			& $3.103\!\times\!10^{-4}$ & $2.0$
			& $8.320\!\times\!10^{-4}$ & $2.0$
			& $7.375\!\times\!10^{-4}$ & $2.0$
			& $2.456\!\times\!10^{-2}$ & $1.1$
			& $6.880\!\times\!10^{-4}$ & $2.0$ \\
			\midrule
			\multicolumn{17}{c}{$\nu_p = 0.499$,\quad
				$\lambda_p = 166.4$,\quad $\kappa = 10^{-6}$} \\
			\midrule
			$h$
			& $\|e_{\bm u_f}\|_{0,\Omf}$ & $r$
			& $\|e_{\bm u_f}\|_{1,\Omf}$ & $r$
			& $\|e_{\bm\eta}\|_{0,\Omp}$ & $r$
			& $\|e_{\bm\eta}\|_{1,\Omp}$ & $r$
			& $\|e_{p_f}\|_{0,\Omf}$ & $r$
			& $\|e_{\xi}\|_{0,\Omp}$ & $r$
			& $\|e_{p_p}\|_{1,\Omp}$ & $r$
			& $\|e_{p_p}\|_{0,\Omp}$ & $r$ \\
			\midrule
			$1/16$
			& $4.343\!\times\!10^{-4}$ & $-$
			& $9.939\!\times\!10^{-3}$ & $-$
			& $1.410\!\times\!10^{-3}$ & $-$
			& $3.073\!\times\!10^{-2}$ & $-$
			& $3.323\!\times\!10^{-2}$ & $-$
			& $3.308\!\times\!10^{-2}$ & $-$
			& $1.005\!\times\!10^{-1}$ & $-$
			& $2.566\!\times\!10^{-3}$ & $-$ \\
			\addlinespace
			$1/32$
			& $5.201\!\times\!10^{-5}$ & $3.1$
			& $2.083\!\times\!10^{-3}$ & $2.3$
			& $1.303\!\times\!10^{-4}$ & $3.4$
			& $5.279\!\times\!10^{-3}$ & $2.5$
			& $6.812\!\times\!10^{-3}$ & $2.3$
			& $6.439\!\times\!10^{-3}$ & $2.4$
			& $2.437\!\times\!10^{-2}$ & $2.0$
			& $3.663\!\times\!10^{-4}$ & $2.8$ \\
			\addlinespace
			$1/64$
			& $5.425\!\times\!10^{-6}$ & $3.3$
			& $4.251\!\times\!10^{-4}$ & $2.3$
			& $1.364\!\times\!10^{-5}$ & $3.3$
			& $1.076\!\times\!10^{-3}$ & $2.3$
			& $1.596\!\times\!10^{-3}$ & $2.1$
			& $1.451\!\times\!10^{-3}$ & $2.1$
			& $7.604\!\times\!10^{-3}$ & $1.7$
			& $5.956\!\times\!10^{-5}$ & $2.6$ \\
			\addlinespace
			$1/128$
			& $5.737\!\times\!10^{-7}$ & $3.2$
			& $8.295\!\times\!10^{-5}$ & $2.4$
			& $2.077\!\times\!10^{-6}$ & $2.7$
			& $2.513\!\times\!10^{-4}$ & $2.1$
			& $5.282\!\times\!10^{-4}$ & $1.6$
			& $5.133\!\times\!10^{-4}$ & $1.5$
			& $5.093\!\times\!10^{-3}$ & $0.6$
			& $2.448\!\times\!10^{-5}$ & $1.3$ \\
			\bottomrule
		\end{tabular}
	}
\end{table}
\noindent
The low-permeability results differ substantially from the
moderate-permeability regime. While the fluid velocity
$\bm u_f$ and the solid displacement $\bm\eta$
continue to exhibit nearly optimal convergence rates,
the pore-pressure approximation deteriorates as the mesh
is refined. For $\nu_p=0.3$, the convergence rate of
$\|e_{p_p}\|_{1,\Omp}$ decreases from approximately
$1.6$ to $1.1$. For the nearly incompressible case
$\nu_p=0.499$, the deterioration is more pronounced,
with reduced and non-monotone convergence behavior
appearing on the finest meshes.
These observations are consistent with
Remark~\ref{rem:param_dependence}. The stability theory
is robust with respect to the nearly incompressible
parameter $\lambda_p$, but retains an explicit
dependence on $\kappa$ through the parameter-weighted
Darcy and BJS dissipation terms. Consequently, when
$\kappa$ is extremely small, pre-asymptotic effects and
reduced pressure convergence may occur without violating
the theoretical stability estimate.
These results are consistent with the parameter-weighted
structure of the stability estimate and suggest that the
low-permeability regime remains numerically challenging.
Although stability is preserved, the results in Table~\ref{tab:robust_kappa01} show that the pressure variables exhibit a noticeable loss of accuracy and reduced
convergence rates as $\kappa$ becomes very small.
Consequently, resolving extremely low-permeability
problems may require finer meshes or additional
discretization techniques to recover the asymptotic
convergence behavior predicted for moderate values of
$\kappa$.
\noindent

We then turn to examine robustness with respect to the storage coefficient $s_0 \geq 0$. The singular limit $s_0 = 0$, in which the storage term vanishes and \eqref{eq:biot_total} becomes a quasi-static constraint with no time derivative
of $p_p$, is the most challenging case.

\begingroup
\setlength{\intextsep}{0pt}
\begin{table}[H]
	\centering
	\caption{Parameter test for $s_0$:
		$s_0 = 0$, $\nu_p = 0.3$, $\kappa = 1$,
		$\alpha = 1$, $\Delta t = 10^{-3}$, $T = 0.5$.}
	\label{tab:robust_s0}
	\footnotesize
	\setlength{\tabcolsep}{3pt}
	\renewcommand{\arraystretch}{1.2}
	\resizebox{\textwidth}{!}{
		\begin{tabular}{c cc cc cc cc cc cc cc cc}
			\toprule
			$h$
			& $\|e_{\bm u_f}\|_{0,\Omf}$ & $r$
			& $\|e_{\bm u_f}\|_{1,\Omf}$ & $r$
			& $\|e_{\bm\eta}\|_{0,\Omp}$ & $r$
			& $\|e_{\bm\eta}\|_{1,\Omp}$ & $r$
			& $\|e_{p_f}\|_{0,\Omf}$ & $r$
			& $\|e_{\xi}\|_{0,\Omp}$ & $r$
			& $\|e_{p_p}\|_{1,\Omp}$ & $r$
			& $\|e_{p_p}\|_{0,\Omp}$ & $r$ \\
			\midrule
			$1/16$
			& $4.316\!\times\!10^{-4}$ & $-$
			& $9.947\!\times\!10^{-3}$ & $-$
			& $1.299\!\times\!10^{-3}$ & $-$
			& $2.705\!\times\!10^{-2}$ & $-$
			& $3.323\!\times\!10^{-2}$ & $-$
			& $3.297\!\times\!10^{-2}$ & $-$
			& $8.561\!\times\!10^{-2}$ & $-$
			& $2.806\!\times\!10^{-3}$ & $-$ \\
			\addlinespace
			$1/32$
			& $5.231\!\times\!10^{-5}$ & $3.0$
			& $2.093\!\times\!10^{-3}$ & $2.2$
			& $1.144\!\times\!10^{-4}$ & $3.5$
			& $4.602\!\times\!10^{-3}$ & $2.6$
			& $6.808\!\times\!10^{-3}$ & $2.3$
			& $6.432\!\times\!10^{-3}$ & $2.4$
			& $2.013\!\times\!10^{-2}$ & $2.1$
			& $3.582\!\times\!10^{-4}$ & $3.0$ \\
			\addlinespace
			$1/64$
			& $5.426\!\times\!10^{-6}$ & $3.3$
			& $4.260\!\times\!10^{-4}$ & $2.3$
			& $1.209\!\times\!10^{-5}$ & $3.2$
			& $9.407\!\times\!10^{-4}$ & $2.3$
			& $1.567\!\times\!10^{-3}$ & $2.1$
			& $1.424\!\times\!10^{-3}$ & $2.2$
			& $4.686\!\times\!10^{-3}$ & $2.1$
			& $3.904\!\times\!10^{-5}$ & $3.2$ \\
			\addlinespace
			$1/128$
			& $5.603\!\times\!10^{-7}$ & $3.3$
			& $8.315\!\times\!10^{-5}$ & $2.4$
			& $1.515\!\times\!10^{-6}$ & $3.0$
			& $2.189\!\times\!10^{-4}$ & $2.1$
			& $3.500\!\times\!10^{-4}$ & $2.2$
			& $3.542\!\times\!10^{-4}$ & $2.0$
			& $1.125\!\times\!10^{-3}$ & $2.1$
			& $4.676\!\times\!10^{-6}$ & $3.1$ \\
			\bottomrule
		\end{tabular}
	}
\end{table}
\endgroup

Table~\ref{tab:robust_s0} reports results for
$s_0 = 0$ with $\nu_p = 0.3$ and $\kappa = 1$.
The results are nearly indistinguishable from those
obtained for $s_0=1$.
All variables retain optimal convergence rates across
the entire refinement sequence, indicating that the
method remains stable and accurate in the singular
storage limit $s_0=0$.
\noindent

We finally investigate the influence of the fluid
viscosity $\mu_f$, which scales the viscous stress
term $2\mu_f\bm D(\bm u_f)$ in the Stokes equations.
As $\mu_f$ decreases, viscous dissipation becomes weaker,
making the fluid subproblem increasingly sensitive to
pressure gradients and interface coupling effects.
To assess the sensitivity of the method to this
parameter, we consider the values
$\mu_f=10^{-1}$ and $\mu_f=10^{-3}$ while keeping
all other parameters fixed.
\begin{table}[H]
	\centering
	\caption{Parameter test for $\mu_f$:
		$\mu_f \in \{10^{-1},\, 10^{-3}\}$,
		$\nu_p = 0.3$, $\kappa = 1$, $s_0 = 1$,
		$\alpha = 1$, $\Delta t = 10^{-3}$, $T = 0.5$.}
	\label{tab:robust_muf}
	\footnotesize
	\setlength{\tabcolsep}{3pt}
	\renewcommand{\arraystretch}{1.2}
	\resizebox{\textwidth}{!}{
		\begin{tabular}{c cc cc cc cc cc cc cc cc}
			\toprule
			\multicolumn{17}{c}{$\mu_f = 10^{-1}$} \\
			\midrule
			$h$
			& $\|e_{\bm u_f}\|_{0,\Omf}$ & $r$
			& $\|e_{\bm u_f}\|_{1,\Omf}$ & $r$
			& $\|e_{\bm\eta}\|_{0,\Omp}$ & $r$
			& $\|e_{\bm\eta}\|_{1,\Omp}$ & $r$
			& $\|e_{p_f}\|_{0,\Omf}$ & $r$
			& $\|e_{\xi}\|_{0,\Omp}$ & $r$
			& $\|e_{p_p}\|_{1,\Omp}$ & $r$
			& $\|e_{p_p}\|_{0,\Omp}$ & $r$ \\
			\midrule
			$1/16$
			& $4.266\!\times\!10^{-3}$ & $-$
			& $9.848\!\times\!10^{-2}$ & $-$
			& $1.298\!\times\!10^{-3}$ & $-$
			& $2.706\!\times\!10^{-2}$ & $-$
			& $3.322\!\times\!10^{-2}$ & $-$
			& $3.297\!\times\!10^{-2}$ & $-$
			& $8.557\!\times\!10^{-2}$ & $-$
			& $2.781\!\times\!10^{-3}$ & $-$ \\
			\addlinespace
			$1/32$
			& $5.207\!\times\!10^{-4}$ & $3.0$
			& $2.086\!\times\!10^{-2}$ & $2.2$
			& $1.143\!\times\!10^{-4}$ & $3.5$
			& $4.602\!\times\!10^{-3}$ & $2.6$
			& $6.808\!\times\!10^{-3}$ & $2.3$
			& $6.433\!\times\!10^{-3}$ & $2.4$
			& $2.013\!\times\!10^{-2}$ & $2.1$
			& $3.588\!\times\!10^{-4}$ & $3.0$ \\
			\addlinespace
			$1/64$
			& $5.419\!\times\!10^{-5}$ & $3.3$
			& $4.257\!\times\!10^{-3}$ & $2.3$
			& $1.195\!\times\!10^{-5}$ & $3.3$
			& $9.406\!\times\!10^{-4}$ & $2.3$
			& $1.567\!\times\!10^{-3}$ & $2.1$
			& $1.424\!\times\!10^{-3}$ & $2.2$
			& $4.686\!\times\!10^{-3}$ & $2.1$
			& $3.906\!\times\!10^{-5}$ & $3.2$ \\
			\addlinespace
			$1/128$
			& $5.435\!\times\!10^{-6}$ & $3.3$
			& $8.313\!\times\!10^{-4}$ & $2.4$
			& $1.384\!\times\!10^{-6}$ & $3.1$
			& $2.189\!\times\!10^{-4}$ & $2.1$
			& $3.499\!\times\!10^{-4}$ & $2.2$
			& $3.542\!\times\!10^{-4}$ & $2.0$
			& $1.125\!\times\!10^{-3}$ & $2.1$
			& $4.602\!\times\!10^{-6}$ & $3.1$ \\
			\midrule
			\multicolumn{17}{c}{$\mu_f = 10^{-3}$} \\
			\midrule
			$h$
			& $\|e_{\bm u_f}\|_{0,\Omf}$ & $r$
			& $\|e_{\bm u_f}\|_{1,\Omf}$ & $r$
			& $\|e_{\bm\eta}\|_{0,\Omp}$ & $r$
			& $\|e_{\bm\eta}\|_{1,\Omp}$ & $r$
			& $\|e_{p_f}\|_{0,\Omf}$ & $r$
			& $\|e_{\xi}\|_{0,\Omp}$ & $r$
			& $\|e_{p_p}\|_{1,\Omp}$ & $r$
			& $\|e_{p_p}\|_{0,\Omp}$ & $r$ \\
			\midrule
			$1/16$
			& $1.301\!\times\!10^{-1}$ & $-$
			& $3.388\!\times\!10^{0}$ & $-$
			& $1.299\!\times\!10^{-3}$ & $-$
			& $2.706\!\times\!10^{-2}$ & $-$
			& $3.452\!\times\!10^{-2}$ & $-$
			& $3.297\!\times\!10^{-2}$ & $-$
			& $8.557\!\times\!10^{-2}$ & $-$
			& $2.780\!\times\!10^{-3}$ & $-$ \\
			\addlinespace
			$1/32$
			& $3.178\!\times\!10^{-2}$ & $2.0$
			& $1.415\!\times\!10^{0}$ & $1.3$
			& $1.143\!\times\!10^{-4}$ & $3.5$
			& $4.602\!\times\!10^{-3}$ & $2.6$
			& $6.863\!\times\!10^{-3}$ & $2.3$
			& $6.433\!\times\!10^{-3}$ & $2.4$
			& $2.013\!\times\!10^{-2}$ & $2.1$
			& $3.589\!\times\!10^{-4}$ & $3.0$ \\
			\addlinespace
			$1/64$
			& $4.832\!\times\!10^{-3}$ & $2.7$
			& $3.903\!\times\!10^{-1}$ & $1.9$
			& $1.195\!\times\!10^{-5}$ & $3.3$
			& $9.406\!\times\!10^{-4}$ & $2.3$
			& $1.567\!\times\!10^{-3}$ & $2.1$
			& $1.424\!\times\!10^{-3}$ & $2.2$
			& $4.686\!\times\!10^{-3}$ & $2.1$
			& $3.906\!\times\!10^{-5}$ & $3.2$ \\
			\addlinespace
			$1/128$
			& $5.101\!\times\!10^{-4}$ & $3.2$
			& $8.137\!\times\!10^{-2}$ & $2.3$
			& $1.384\!\times\!10^{-6}$ & $3.1$
			& $2.189\!\times\!10^{-4}$ & $2.1$
			& $3.499\!\times\!10^{-4}$ & $2.2$
			& $3.542\!\times\!10^{-4}$ & $2.0$
			& $1.125\!\times\!10^{-3}$ & $2.1$
			& $4.603\!\times\!10^{-6}$ & $3.1$ \\
			\bottomrule
		\end{tabular}
	}
\end{table}
The results in Table~\ref{tab:robust_muf} show that the convergence behavior of the method is largely unaffected by variations in
the fluid viscosity. For $\mu_f=10^{-1}$, all variables exhibit the same spatial convergence rates observed in the reference
configuration.

For $\mu_f=10^{-3}$, the magnitude of the fluid
velocity error increases, particularly in the
$H^1(\Omega_f)$ norm, reflecting the reduced
viscous regularization of the Stokes subsystem.
Nevertheless, the convergence rates recover under
mesh refinement and approach their asymptotic values.
The poroelastic variables remain essentially
unchanged, indicating that the coupled formulation
does not exhibit sensitivity to small fluid viscosity.

Taken together, the parameter studies demonstrate that the proposed monolithic BDF2 total-pressure formulation is highly robust with respect to the material and coupling parameters governing the Stokes--Biot system. The convergence behavior remains essentially unchanged over a wide range of values of the Lam\'e parameter $\lambda_p$, including the nearly incompressible limit $\nu_p \rightarrow 0.5$, confirming the locking-free character of the total-pressure formulation. Similarly, the method retains optimal spatial convergence in the singular storage limit $s_0=0$, indicating robustness with respect to fluid compressibility. Variations in the fluid viscosity $\mu_f$ affect the magnitude of the fluid velocity error, particularly for very small viscosities, but do not alter the asymptotic convergence behavior of either the fluid or poroelastic variables. In contrast, the permeability parameter $\kappa$ has a more pronounced influence. For moderate permeabilities, such as $\kappa=10^{-1}$ and $\kappa=1$, all variables exhibit the predicted convergence rates. However, when $\kappa$ is reduced to $10^{-6}$, the pore-pressure approximation deteriorates, and optimal convergence is no longer observed, while the remaining variables continue to converge satisfactorily. This behavior is consistent with the singularly perturbed nature of the Darcy operator in the low-permeability regime and suggests that additional mesh refinement or adaptive resolution may be required to accurately capture the resulting pore-pressure boundary layers. Overall, the numerical results provide strong support for the theoretical stability and convergence analysis. The experiments confirm robustness with respect to $\lambda_p$, $s_0$, and $\mu_f$, while indicating that additional spatial resolution may be required in the extreme low-permeability regime to accurately resolve the pore-pressure solution.
\section{Conclusions}
\label{sec:conclusions}
In this work, we developed and analyzed a fully implicit monolithic BDF2 finite element method for the coupled time-dependent Stokes--Biot system in a three-field total-pressure formulation. The proposed method advances the fluid velocity, fluid pressure, solid displacement, total pressure, and pore pressure simultaneously at each time step while enforcing all interface conditions implicitly. In contrast to partitioned approaches, the formulation does not require operator splitting, interface extrapolation, or time-step restrictions.

The theoretical analysis established discrete energy stability of the fully discrete scheme through the BDF2 $G$-stability framework. Stability estimates were shown to be robust with respect to the physical parameters, including the Lam\'e parameter, the Biot--Willis coefficient, and the storage coefficient. Furthermore, optimal a priori error estimates were derived, establishing second-order convergence in time together with optimal spatial convergence.

The numerical experiments strongly support the theoretical results. Temporal convergence studies confirmed the expected second-order accuracy of the BDF2 discretization, while spatial convergence studies recovered the optimal finite element rates predicted by the analysis. Additional parameter studies demonstrated robustness with respect to variations in the Lam\'e parameter, including the nearly incompressible regime, as well as robustness with respect to the storage coefficient and fluid viscosity. For moderate permeability values, the method maintained its optimal convergence behavior and stability properties. In the extreme low-permeability regime, however, deterioration of the pore-pressure approximation was observed, indicating the presence of unresolved pressure layers and highlighting the increased spatial resolution required in singularly perturbed Darcy-dominated problems.

Overall, the results demonstrate that the proposed monolithic BDF2 total-pressure formulation provides a stable, accurate, and reliable framework for the numerical simulation of coupled fluid--poroelastic interaction problems across a broad range of physically relevant parameter regimes. The combination of rigorous analysis and comprehensive numerical verification confirms its effectiveness for challenging Stokes--Biot applications.

Future work will focus on the development of efficient parameter-robust preconditioners with Multigrid solvers \cite{ChenChen2021} for the fully coupled algebraic systems, mortar methods that use non-matching meshes in different subdomains \cite{Ambartsumyan2018, Huang2017mortar}, adaptive mesh refinement strategies for heterogeneous and low-permeability media, higher-order temporal discretizations, and extensions to nonlinear and multiphysics poroelastic models arising in geomechanics, biomechanics, and subsurface flow applications.

\section*{Acknowledgements}
The work of TA and MC is partially supported by the I-GAP Small Tech Transfer Grant I-GAP OTT-STT-272, sponsored by the Office of Technology Transfer at Morgan State University, and the support of The University of Maryland Baltimore Life Science Discovery (UM-BILD) Accelerator and REACH Hub Award 1U01GM152511-03.


\begin{thebibliography}{99}
	
	\bibitem{Ambartsumyan2018}
	I.~Ambartsumyan, E.~Khattatov, I.~Yotov, and P.~Zunino,
	\newblock A Lagrange multiplier method for a Stokes--Biot fluid--poroelastic structure interaction model,
	\newblock {\em Numerische Mathematik}, 140(2):513--553, 2018.
	
	\bibitem{BadiaQuainiQuarteroni2009}
	S.~Badia, A.~Quaini, and A.~Quarteroni,
	\newblock Coupling Biot and Navier--Stokes equations for modelling fluid--poroelastic media interaction,
	\newblock {\em Journal of Computational Physics}, 228(21):7986--8014, 2009.
	
	\bibitem{Biot1941}
	M.~A.~Biot,
	\newblock General theory of three-dimensional consolidation,
	\newblock {\em Journal of Applied Physics}, 12(2):155--164, 1941.
	
	\bibitem{Biot1956}
	M.~A.~Biot,
	\newblock Theory of propagation of elastic waves in a fluid-saturated porous solid,
	\newblock {\em Journal of the Acoustical Society of America}, 28(2):168--178, 1956.
	
	\bibitem{BoffiBrezziFortin2013}
	D.~Boffi, F.~Brezzi, and M.~Fortin,
	\newblock {\em Mixed Finite Element Methods and Applications},
	\newblock Springer, 2013.
	
	\bibitem{BukacYotov2015}
	M.~Buka\v{c}, I.~Yotov, R.~Zakerzadeh, and P.~Zunino,
	\newblock Partitioning strategies for the interaction of a fluid with a poroelastic material,
	\newblock {\em Computer Methods in Applied Mechanics and Engineering}, 292:138--170, 2015.
	
	\bibitem{CesmeliogluChidyagwai2020}
	A.~Cesmelioglu and P.~Chidyagwai,
	\newblock Numerical analysis of the coupling of free fluid with a poroelastic material,
	\newblock {\em Numerical Methods for Partial Differential Equations}, 36(3):463--494, 2020.
	
	\bibitem{Chen2013}
	W.~Chen, M.~Gunzburger, D.~Sun, and X.~Wang,
	\newblock Efficient and long-time accurate second-order methods for the Stokes--Darcy system,
	\newblock {\em SIAM Journal on Numerical Analysis}, 51(5):2563--2584, 2013.
	
	\bibitem{ChenChen2021}
	L.~Chen and Y.~Chen,
	\newblock Multigrid method for poroelasticity problem by finite element method,
	\newblock {\em Advances in Applied Mathematics and Mechanics},
	2021.
	doi: \href{https://doi.org/10.4208/aamm.OA-2019-0003}
	{10.4208/aamm.OA-2019-0003}.
	
	\bibitem{DiscacciatiQuarteroni2009}
	M.~Discacciati and A.~Quarteroni,
	\newblock Navier--Stokes/Darcy coupling: modeling, analysis, and numerical approximation,
	\newblock {\em Revista Matemática Complutense}, 22(2):315--426, 2009.
	
	\bibitem{GiraultRaviart1986}
	V.~Girault and P.-A.~Raviart,
	\newblock {\em Finite Element Methods for Navier--Stokes Equations: Theory and Algorithms},
	\newblock Springer-Verlag, Berlin, 1986.
	
	\bibitem{gu2023iterative}
	H.~Gu, M.~Cai, and J.~Li,
	\newblock An iterative decoupled algorithm with unconditional stability for Biot model,
	\newblock {\em Mathematics of Computation}, 92(341):1087--1108, 2023.
	
	\bibitem{gu2025convergence}
	H.~Gu, M.~Cai, and J.~Li,
	\newblock Convergence analysis of a global-in-time iterative decoupled algorithm for Biot’s model,
	\newblock {\em Advances in Applied Mathematics and Mechanics},
	17(3):778--803, 2025.
	
	\bibitem{GuoLi2026}
	L.~Guo and S.~Li,
	\newblock Stability and error estimates of a second-order decoupled FEM for Stokes--Biot,
	\newblock {\em Communications in Nonlinear Science and Numerical Simulation}, 158:109806, 2026.
	
	\bibitem{HairerWanner1996}
	E.~Hairer and G.~Wanner,
	\newblock {\em Solving Ordinary Differential Equations II},
	\newblock Springer, 1996.
	
	\bibitem{HeGuoFeng2024}
	L.~He, J.~Guo, and M.~Feng,
	\newblock Analysis of two discontinuous Galerkin finite element methods for the total pressure formulation of linear poroelasticity model,
	\newblock {\em Applied Numerical Mathematics},
	204:60--85, 2024.
	
	\bibitem{Heywood1990}
	J.~G.~Heywood and R.~Rannacher,
	\newblock Finite element approximation of the nonstationary Navier--Stokes problem IV,
	\newblock {\em SIAM Journal on Numerical Analysis}, 27(2):353--384, 1990.
	
	\bibitem{Huang2017mortar}
	P.~Huang, J.~Chen, and M.~Cai,
	\newblock A mortar method using nonconforming and mixed finite elements for the coupled Stokes--Darcy model,
	\newblock {\em Advances in Applied Mathematics and Mechanics},
	9(3):596--620, 2017.
	
	\bibitem{Lee2017}
	J.~J.~Lee, K.-A.~Mardal, and R.~Winther,
	\newblock Parameter-robust discretization and preconditioning of Biot's consolidation model,
	\newblock {\em SIAM Journal on Scientific Computing}, 39(1):A1--A24, 2017.
	
	\bibitem{Murad2004}
	M.~A.~Murad and A.~F.~D.~Loula,
	\newblock Improved accuracy in finite element analysis of Biot's consolidation problem,
	\newblock {\em Computer Methods in Applied Mechanics and Engineering}, 193(9--11):965--981, 2004.
	
	\bibitem{OyarzuaRuizBaier2016}
	R.~Oyarz\'{u}a and R.~Ruiz-Baier,
	\newblock Locking-free finite element methods for poroelasticity,
	\newblock {\em SIAM Journal on Numerical Analysis}, 54(5):2951--2973, 2016.
	
	\bibitem{OyekoleBukac2020}
	O.~Oyekole and M.~Buka\v{c},
	\newblock Second-order, loosely coupled methods for fluid--poroelastic interaction,
	\newblock {\em Numerical Methods for Partial Differential Equations}, 36(4):800--822, 2020.
	
	\bibitem{PhillipsWheeler2007}
	P.~J.~Phillips and M.~F.~Wheeler,
	\newblock A coupling of mixed and continuous Galerkin finite element methods for poroelasticity I,
	\newblock {\em Computational Geosciences}, 11:131--144, 2007.
	
	\bibitem{PhillipsWheeler2009}
	P.~J.~Phillips and M.~F.~Wheeler,
	\newblock A coupling of mixed and continuous Galerkin finite element methods for poroelasticity II,
	\newblock {\em Computational Geosciences}, 13:145--160, 2009.
	
	\bibitem{RuizBaier2022}
	R.~Ruiz-Baier, M.~Taffetani, H.~D.~Westermeyer, and I.~Yotov,
	\newblock The Biot--Stokes coupling using total pressure,
	\newblock {\em Computer Methods in Applied Mechanics and Engineering}, 389:114384, 2022.
	
	\bibitem{Thomee2006}
	V.~Thom\'ee,
	\newblock {\em Galerkin Finite Element Methods for Parabolic Problems},
	\newblock Springer, 2006.
	
	\bibitem{ZhouLiChen2024}
	M.~Zhou, R.~Li, and Z.~Chen,
	\newblock A discontinuous Galerkin method for a coupled Stokes--Biot problem,
	\newblock {\em Journal of Computational and Applied Mathematics}, 451:116086, 2024.
	
	\bibitem{ZhouLiZhangChen2026}
	B.~Zhou, R.~Li, C.-S.~Zhang, and Z.~Chen,
	\newblock Combined discontinuous and continuous Galerkin methods for fractured poroelastic media flow on polytopic grids,
	\newblock {\em Journal of Computational and Applied Mathematics}, 474:116986, 2026.
	
\end{thebibliography}
\end{document}